\documentclass[letter,11pt]{amsart}
\usepackage[toc,page]{appendix}
\usepackage{amssymb, amsfonts} 
\usepackage{graphicx}
\usepackage{color}
\usepackage[font=footnotesize]{caption}
\usepackage{mathtools}
\usepackage{pinlabel}
\usepackage{graphicx}
\usepackage{tikz}
\usepackage{tikz-cd}
\graphicspath{ {images/} }
\usepackage{caption}
\usepackage{subcaption}
\usepackage{multirow}
\usepackage{longtable}
\usepackage{adjustbox}

\usepackage{geometry}
\usepackage{hyperref}

\usepackage{float}
\usepackage{calc}
\def\thetitle{{ . }}
\hypersetup{
	pdftitle=  \thetitle,
	pdfauthor=  {Juhun Baik}
}

\newtheorem{thm}{Theorem}[section]
\newtheorem{lem}[thm]{Lemma}
\newtheorem{cor}[thm]{Corollary}
\newtheorem{prop}[thm]{Proposition}
\newtheorem{que}{Question}

\newtheorem{defn}[thm]{Definition}
\newtheorem*{que*}{Question}

\theoremstyle{remark}

\newtheorem*{rmk}{\textbf{Remark}}

\usepackage{tikz}
\usetikzlibrary{calc,decorations.markings}
\usetikzlibrary{shapes,snakes}

\theoremstyle{definition}
\newtheorem*{defn*}{Definition}

\newcommand\Mod{\operatorname{Mod}}

\newcommand\BDK{\operatorname{BDK}}
\newcommand\PER{\operatorname{PER}}

\title{Monodromy Through Bifurcation Locus of the Mandelbrot Set}

\author{Hyungryul Baik}
\address{Hyungryul Baik, Department of Mathematical Sciences, KAIST,  
	291 Daehak-ro, Yuseong-gu, Daejeon 34141, South Korea }
\email{hrbaik@kaist.ac.kr}

\author{Juhun Baik}
\address{Juhun Baik, Department of Mathematical Sciences, KAIST,  
	291 Daehak-ro, Yuseong-gu, Daejeon 34141, South Korea }
\email{jhbaik@kaist.ac.kr}

\date{\today}

\begin{document}
	
	\begin{abstract}
		We investigate the behavior of itinerary sequence of each point of the Julia set of $z\mapsto z^2 + c$ when the parameter $c$ in the shift locus is allowed to pass through points in the bifurcation locus $\mathcal{P}_2$, which we call ``narrow", first proposed by Dierk Schleicher in \cite{schleicher2017internal}.
		We first show the combinatoric and geometric properties of narrow characteristic arcs.
		Also, we show how the itinerary sequence changes in an algorithmic way by using lamination models proposed by Keller in \cite{keller2007invariant}.
		Finally, we found an equivalence relation on the set of $0$-$1$ sequences so that the changing rule is a shift invariant up to the equivalence relation.
		This generalizes Atela's works in \cite{atela1992bifurcations}, \cite{atela1993mandelbrot}, which dealt with the special case of the generalized rabbit polynomials.
	\end{abstract}
	\maketitle
	
	\section{Introduction} \label{sec:1-intro} 
	The shift locus of complex polynomials of degree $d\geq 2$ is a collection of polynomials that every critical point escapes to infinity under iterations of itself.
	The reason we call it a shift polynomial is the following theorem.
	\begin{thm}\label{thm:shift_defn}
		Suppose $P$ is a shift polynomial of degree $d\geq 2$ and $J_P$ be a Julia set of $P$.
		Then there exists a homeomorphism between $J_P$ and $\Sigma_d$, a set of one-sided infinite sequence of $d$ symbols.
		Furthermore $P \vert_{J_P} : J_P \to J_P$ is conjugated by this homeomorphism to the one-sided shift map $\sigma: \Sigma_d \to \Sigma_d$.
		\emph{i.e.,} 
		\[
		\begin{tikzcd}
			J_P \arrow[r] \arrow[d, "P"] & \Sigma_d \arrow[d, "\sigma"] \\
			J_P \arrow[r]        & \Sigma_d                               
		\end{tikzcd}
		\]
	\end{thm}
	We endow the topology of $\Sigma_d$ as the product topology of $\{0,1,\cdots, d-1\}^\mathbb{N}$, hence it is homeomorphic to a Cantor set.
	
	The most famous and the simplest one is the exterior of Mandelbrot set, $\mathbb{C} - \mathcal{M}$, which is the shift locus of quadratic polynomials $\mathcal{S}_2$.
	As the Mandelbrot set seems fractal, a shift locus of degree $d$, $\mathcal{S}_d$, has also a fractal structure and it has been a challenge to understand its topology and geometry.
	In 1994, Blanchard, Devaney, and Keen  \cite{blanchard1991dynamics} showed that there is an action by the fundamental group of the shift locus $\pi_1(\mathcal{S}_d)$ on the set of shift automorphisms of $d$ symbols.
	Furthermore they proved that this action is surjective. 
	\emph{i.e.,} for any shift automorphism $\phi$, there exists an element in $\pi_1(\mathcal{S}_d)$ which acts on the Julia set exactly same as $\phi$.
	\[
    	\BDK : \pi_1(\mathcal{S}_d) \twoheadrightarrow Aut(\Sigma_d, \sigma)
	\]
	Here a shift automorphism is an automorphism of $\Sigma_d$ which commutes with the one-sided shift map $\sigma$.
	
	Unlike $d\geq 3$, the quadratic case has no interests in usual, because the shift locus $\mathcal{S}_2$ is conformal to $\mathbb{C} - \mathbb{D}$ as proved by Douady and Hubbard.
	Hence $\pi_1(\mathcal{S}_2)$ is $\mathbb{Z}$ and it is generated by a loop $\gamma$ which wraps around $\mathcal{M}$ as in the figure \ref{fig:Mandelbrot_and_gamma}.
	Under $\BDK$ map, this generator acts on $\Sigma_2$ as a symbol change of $0$ and $1$\, which is the only nontrivial shift automorphism in $\Sigma_2$.
	\begin{align*}
		\BDK : \pi_1(\mathcal{S}_2)~\cong~\langle\gamma\rangle~ & \longrightarrow Aut(\Sigma_2, \sigma)\\
		\gamma \qquad&\longmapsto (0,1) \text{ symbol swap}
	\end{align*}
	Let $J_\alpha$ be a Cantor Julia set of $P_\alpha : z\mapsto z^2 + c$, where the parameter angle of $c$ is $\alpha$.
	Then as $c$ moves around the loop, $J_0$ permutes itself and came back to the position where they start.
	
	\begin{figure}[h]
		\centering
		\includegraphics[width = .5\textwidth]{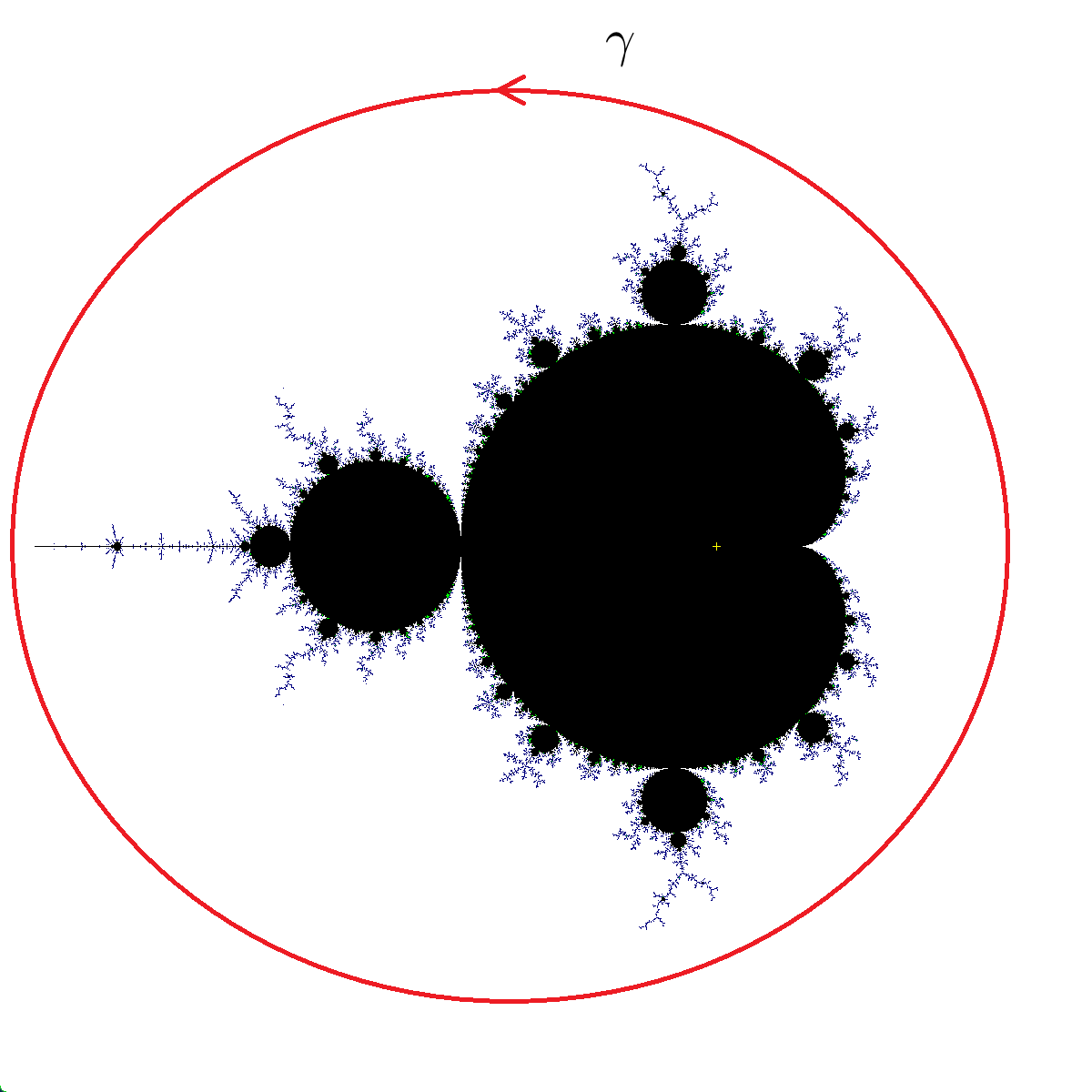}
		\caption{$\gamma$ is a loop which moves around the Mandelbrot set $\mathcal{M}$. $\gamma$ is a generator of the shift locus $\mathcal{S}_2 \cong \mathbb{C} - \overline{\mathbb{D}}$.}
		\label{fig:Mandelbrot_and_gamma}
	\end{figure}

    \begin{figure}
        \centering
        \includegraphics[width = .7\textwidth]{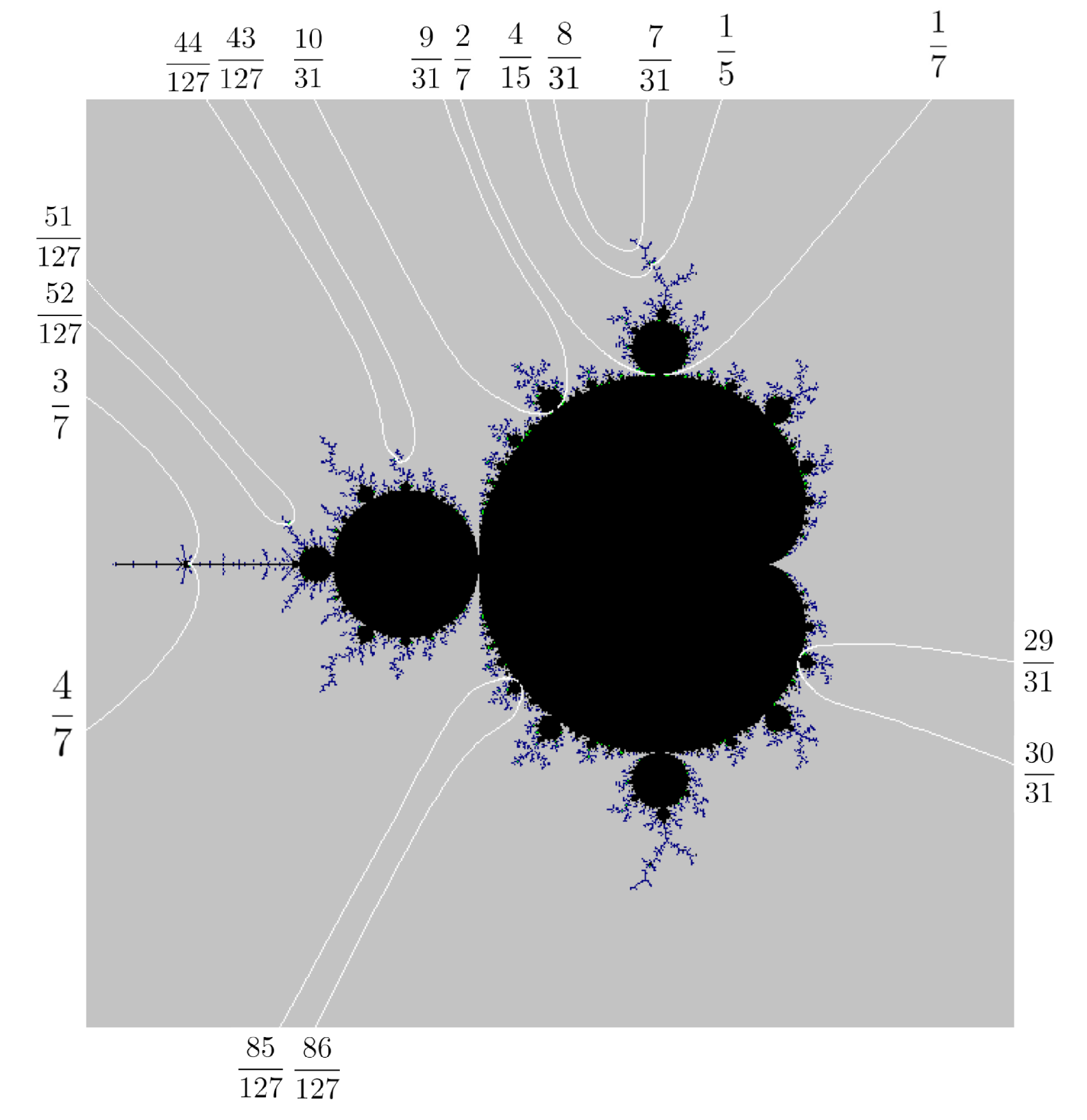}
        \caption{Examples of parameter ray pairs landing at the bifurcation locus, which are all narrow. We investigate what happens when $c$ moves along a pair of parameter rays.}
        \label{fig:bifurcations}
    \end{figure}
    
	\subsection*{Narrow characteristic arcs}
	Milnor proved in \cite{milnor2000periodic} that	there is a combinatorial way (called ``Orbit portrait") of describing the dynamical properties of parabolic points.
	Also, Douady, Hubbard and Lavarus proved in \cite{douady1984etude} that every parabolic point has exactly two parameter rays landing at the root of a hyperbolic component. 
	Those pair of angles are called \emph{companion angles}, denoted as $\alpha$ and $\overline{\alpha}$.
	The arc whose endpoints form a pair of companion angles is called a \emph{characteristic arc}.
	The collection of all characteristic arcs became a lamination, so called \emph{Quadratic Minor Lamination}.
	In \cite{schleicher2017internal}, Schleicher proposed the notion of narrow characteristic arcs.
    Simply put, a narrow arc of period $n$ is an arc of width $\frac{1}{2^{n}-1}$.
	We give a necessary and sufficient condition of narrowness for a characteristic arc in terms of its orbit portrait.
	By using this, we propose a geometric condition of simply renormalizable arcs.
	\begin{thm}[Corollary \ref{cor:simply_renormalizable_criterion}]
		Every narrow arc/component is not simply renormalizable.
		Every prime period arc/component is not simply renormalizable.
		Every satellite arc/component is not narrow unless its root meets in the main cardioid $M_0$.
	\end{thm}
	
	We remark that Schleicher define the \emph{internal address} of each hyperbolic component of $\mathcal{M}$ in \cite{schleicher2017internal}.
	By using that he completely determine when the corresponding polynomial is renormalizable, including both simple and crossed cases.
	Thanks to his theorem we get the collection of non-simply renormalizable polynomials which is not narrow with non-prime period in Appendix \ref{appendix:B-3non}.
	
	\subsection*{Kneading sequence changes at the bifurcation locus}	
	The set of external angles for all points in Julia set $J_\alpha$ is equal to that of $J_{\overline{\alpha}}$.
	Hence in the kneading/itinerary sequence sense it is natural to ask what happens if we change $\alpha$ to $\overline{\alpha}$.
    Figure \ref{fig:bifurcations} shows some examples of parameter rays which are narrow.
	We proved for narrow cases there is an algorithm how the itinerary sequence changes in theorem \ref{thm:rule_unified}.

	In \cite{atela1993mandelbrot} and \cite{atela1992bifurcations}, Atela obtained such an algorithm for the case of generalized rabbit polynomials (in the paper, Atela called it \emph{main bifurcation points}.)
    Generalized rabbit polynomials or main bifurcation points correspond to the set of characteristic arcs of parameter angle $\left( \frac{1}{2^{n}-1}, \frac{2}{2^{n}-1} \right)$ for $n\geq 2$.
	Atela constructs the dynamic graph and every itinerary sequence is lifted to the infinite path of the graph, and find the rule by differing the coloring of a pair of vertices.
	We generalized this result to every narrow characteristic arc. 
    Our method is different from Atela's in that we use the lamination model proposed by Keller in \cite{keller2007invariant}. 
    Via this method, we can describe more explicitly how itinerary sequences change.
	
	\begin{thm}[Theorem \ref{thm:rule_unified}]
		Let $s = s_1s_2s_3\cdots$ be a $0$-$1$ sequence.
		Also, let $v = v^\alpha$ be the repetition word of $\alpha$, $e = e^\alpha$ is a characteristic symbol of $\alpha$ and $\PER(\alpha) = |v|+1$ be the period of $\alpha$.
		Then we define $\varphi_\alpha(s)$ as below.
		\begin{enumerate}
			\item $i = 0$.
			\item Set $w = s_i\cdots s_{i+|v|}$.
			\begin{enumerate}
				\item If $w = 0v$ or $1v$, set $i \leftarrow i+|v|+1$ and go back to (2).
				\item If not, but if there is a previous $0v$ or $1v$, check $s_i = e$. 
				\begin{enumerate}
					\item If $s_{i+|v|+1} = e$, swap all $0v$ and $1v$. 
					\item If $s_{i+|v|+1} \neq e$, trace back to the previous $0v$ (or $1v$) and go back to (b). 
				\end{enumerate}
			\end{enumerate}			
			\item If $w$ is neither $0v$ nor $1v$, then set $i \leftarrow i+1$ and go back to (2).
		\end{enumerate}

        Suppose $(\alpha, \overline{\alpha})$ is narrow.
        Choose $p \in J_\alpha$ and let $\theta$ be the dynamic angle of $p$.
        If $\theta$ is not precritical angle, then $\varphi_\alpha(I^\alpha(\theta)) = I^{\overline{\alpha}}(\theta)$.
	\end{thm}
	The power of the theorem is that (1) we do not have to lift the given sequence to an infinite path of the graph as proposed in \cite{atela1993mandelbrot}, and (2) we can always get the target itinerary sequence by just looking at the given sequence from the beginning.
	
	As in \cite{atela1993mandelbrot}, we also provide the dynamical graph method to understand $\varphi_\alpha$ in the appendix \ref{appendix:A-sym_dyn}.
    Such method is applicable regardless of narrowness of $\alpha$, though it is hard to get the target sequence explicitly.

    At this moment, the algorithm $\varphi_\alpha$ only covers the points which are not precritical angles.
    To extend the domain to the whole $J_\alpha \cong \{0,1\}^\mathbb{N}$, we give a marking on each precritical angle, which will be introduced in the next subsection.
	
	\subsection*{Equivalence relation on itinerary sequences}
	The itinerary sequence function $I^\alpha : S^1 \to \Sigma_2$ is not surjective and moreover not defined for some angles, which are preimages of $\alpha$ under angle doubling map.
	Try to make each itinerary sequence compatible with the points in Cantor type Julia set, we first extend the angle of $S^1$ by $\mathbf{E}_\alpha$ (see Definition \ref{defn:extended_angle}).
	
	Furthermore, $\varphi_\alpha$ is neither injective nor surjective, and sometimes multi-valued.
	We investigate when such behaviors happen and resolve the issues by defining the equivalence relation as follows.
	
	\begin{defn}[Definition \ref{defn:equivalence_relation_on_Sigma_2}]
		Let $s_1, s_2 \in \Sigma_2$ and $x_1, x_2 \in J_\alpha$ be the corresponding point.
		Suppose $\theta_1, \theta_2$ are extended angle of $x_1, x_2$, respectively.
		We define the equivalence relation $\sim$ as follows,
		\[
		s_1 \sim s_2 \quad \Leftrightarrow \quad \theta_1 \approx_\alpha \theta_2
		\]
	\end{defn}
	Quotienting $\mathbf{E}_\alpha$ by this equivalence relation, $\varphi_\alpha$ became well-defined and shift-invariant. Here the shift-invariant means that it commutes with one-sided shift of sequences.
	
	\begin{thm}[Theorem \ref{thm:shift_invariant}]
		$\varphi_\alpha : \Sigma_2 / \sim~ \to \Sigma_2 / \sim$ is well-defined.
		Moreover it is order $2$ element and shift-invariant. 
	\end{thm}
	
	\subsection*{Future works}
	In the upcoming paper, we will construct big mapping classes in $\Mod(\mathbb{C} - \{ \text{Cantor set}\})$ or $\Mod(S^2 - \{ \text{Cantor set}\})$ associated to each point in the bifurcation locus of degree $2$.
	Such big mapping classes generate a subgroup of $\Mod(\mathbb{C} - \{ \text{Cantor set}\})$ or $\Mod(S^2 - \{ \text{Cantor set}\})$ which features many interesting dynamical properties. 
	
	\subsection*{Outline of the paper}
	We start with a brief introduction of basic complex dynamics in section \ref{sec:2-prelim}.
	In section \ref{sec:3-char_arcs}, narrow arcs and combinatorics of its orbit portrait are discussed with details and we prove that the narrow condition is sufficient to be not simply renormalizable.
	In section \ref{sec:4-inv_lam} we introduce Keller's works in \cite{keller2007invariant} and define the quadrilateral lamination to describe how the itinerary sequence changes, only for angles which are not pre-critical angle under angle doubling map.
	Section \ref{sec:5-action_on_iti_seq} cares the pre-critical angles to completely define $\varphi_\alpha$ for critical angle $\alpha$.
	We extend angles in $S^1$ with some markers and in section \ref{sec:6-equiv_rel} we give an equivalence relation to make $\varphi_\alpha$ to be shift-invariant.
	In appendix \ref{appendix:A-sym_dyn} we show that for narrow arcs there is a dynamic graph analogous to the one in \cite{atela1993mandelbrot}.
	We end with the table of non-narrow, non simply renormalizable with non-prime period arcs, attached in appendix \ref{appendix:B-3non}.
	
	\subsection*{Acknowledgement}
	We thanks to InSung Park for helpful conversations.
    Also, we thanks to the program \emph{mandel} \cite{jung21mandel} by Wolf Jung which helps to understand and visualize a lot of examples.

	\section{Preliminary} \label{sec:2-prelim}
	In this section we briefly summarize basic theories related to the dynamics of quadratic polynomials.
	For more details, we recommend \cite{milnor2011dynamics}, \cite{hubbard2016teichmuller}, \cite{douady1984exploring} and \cite{carleson2013complex}.
    Some figures are drawn by the computer program \emph{`mandel'}, \cite{jung21mandel} written by Wolf Jung.
	\subsection{Quadratic polynomials}~
 
	Any complex polynomial $a_dz^d + \cdots + a_0$ can be conjugated to a \emph{monic} ($a_d = 1$) and \emph{centered} ($a_{d-1} = 0$) polynomial by a linear map. In particular, any quadratic polynomial $P(z) = \alpha z^2 + \beta z + \gamma$ is conjugate to a polynomial of the form $z\mapsto z^2 + c$ for some $c\in \mathbb{C}$, hence quadratic polynomials are determined up to linear conjugacy by the constant term $c\in \mathbb{C}$. Consequently the moduli space of quadratic polynomials are just a complex plane $\mathbb{C}$.
	We first briefly recall the definition of Fatou set and Julia set.
	\begin{defn}[Fatou set, Julia set]\label{defn:Fatou_Julia}
		Let $P : S \to S$ be a quadratic polynomial which maps the Riemann sphere $S$ to itself.
		Fatou set $F_P$ is a collection of points $z\in S$ such that there is a neighborhood $U$ of $z$ satisfying that the iterations of $P$ at $U$, $P\vert_U, P^{\circ 2}\vert_U, \cdots$ became a normal family. 
		By definition, Fatou set is open.
		Julia set $J_P$ is the complement of $F_P$.
		\emph{i.e.,} $J_P = S - F_P$.
	\end{defn}
	Also, the set of points whose forward orbit $\{P^{\circ n}(z)\}_{n \in \mathbb{Z}}$ is bounded is called the filled Julia set $K_P$. As the name suggests, the boundary of $K_P$ is $J_P$.
	The connectedness of Julia set is determined by the behavior of critical point $z = 0$.
	\begin{thm}\label{thm:shift=escape}
		Let $P(z) = z^2 + c$, $c\in S$.
		Then 
		\begin{itemize}
			\item $J_P$ is connected $\Leftrightarrow$ $\{P^{\circ n}(0)\}_{n = 1, 2, \cdots}$ is bounded.
			\item $J_P$ is a Cantor set $\Leftrightarrow$ $\{P^{\circ n}(0)\}_{n = 1, 2, \cdots}$ escapes to infinity.
		\end{itemize}
		Furthermore if $0$ escapes to infinity under iteration of $P$ (the latter case), $P\vert_{J(P)}$ is conjugate to the one-sided shift on $2$ symbols.
	\end{thm}
	The last statement is exactly the theorem \ref{thm:shift_defn} for the case $d = 2$.
	The shift locus of degree $2$, denoted as $\mathcal{S}_2$, is the collection of all quadratic shift polynomials.
	Recall that the Mandelbrot set $\mathcal{M}$ is the set of $c\in \mathbb{C}$ that $0$ does not escape to infinity, i.e., it is the connectedness locus by the theorem \ref{thm:shift=escape}. 
	Therefore, $\mathcal{S}_2$ is the complement of $\mathcal{M}$.
	
	For a quadratic polynomial $P$, suppose some $z$ satisfies $P^{\circ k}(z) = z$.
	We call the smallest such positive integer $k$ is called a period and $\{z, P(z), \cdots, P^{\circ (k-1)}(z)\}$ as a periodic cycle (abusing the notation,  sometimes it is just called a fixed point even though it is not a single point). 
	We define the multiplier $\rho(z) := \big(P^{\circ k}\big)'(z) = P'(z)\times P'(P(z)) \times\cdots P'(P^{\circ (k-1)}(z))$ for each periodic cycle of period $k$.	Note that the multiplier stays constant in the same periodic cycle.
	Fixed points are classified by their multipliers: each fixed point is called superattracting if $|\rho| = 0$, attracting if $0 < |\rho| < 1$, irrationally indifferent if $\rho = e^{2\pi i \theta}$ with irrational $\theta$, parabolic if $\rho = e^{2\pi i \theta}$ with rational $\theta$, and repelling if $|\rho| > 1$, respectively. 
	\begin{defn}[Hyperbolic component]\label{defn:Hyp_comp}
		Let $P_c := z^2 + c$ for $c\in \mathbb{C}$. 
		Denote $X_k$ be the set of pairs $(c, z) \in \mathbb{C}^2$ such that $P_c^{\circ k}(z) = z$. 
		Consider the function $\rho_k : X_k \to \mathbb{C}$ which sends $(c, z) \mapsto (P_c^{\circ k})'(z)$, the multiplier of such fixed point.
		Let $A_k$ be the set of pairs $(c, z) \in X_k$ with $|\rho(c, z)| < 1$ and $M_k$ be $\pi(A_k)$, where $\pi : (c, z)\mapsto c$ is the projection onto the first factor (\emph{i.e.,} $M_k$ is the collection of $c$ such that $P_c$ has an attracting fixed point of period $k$).
		Then each component $M$ of $M_k$ is an open component of the interior $\mathring{\mathcal{M}}$ of $\mathcal{M}$, and called ``\emph{hyperbolic component}".
	\end{defn}
	It is still an open question that whether every component of $Int(\mathcal{M})$ is hyperbolic or not. 
	This is called \emph{hyperbolic density conjecture}.
	
    The multiplier map $\rho$ gives an analytic isomorphism from  each hyperbolic component $M$ to an open unit disc $\mathbb{D}$ and it extends to a homeomorphism from $\overline{M}$ to $\overline{\mathbb{D}}$.
	Abusing the notation, let $\rho$ be the extended homeomorphism.
	Then $c_M := \rho^{-1}(0)$ is called the \emph{center} of $M$ and $r_M := \rho^{-1}(1)$ is called the \emph{root} of $M$.
	
	\subsection{Parameter angle}~
	
	Douady and Hubbard proved in \cite{douady1984exploring} that $\mathcal{M}$ is connected, by explicitly presenting an analytic map between $\mathbb{C} - \overline{\mathbb{D}}$ to $\mathbb{C} - \mathcal{M}$.
	It is not yet known whether $\mathcal{M}$ is locally connected and this is called \emph{MLC conjecture}.
	
	\begin{thm}\label{thm:Mandelbrot_is_connected}
		There is an analytic isomorphism $\Psi$ between $\mathbb{C} - \mathcal{M}$ and $\mathbb{C} - \overline{\mathbb{D}}$, and hence the Mandelbrot set $\mathcal{M}$ is connected.
	\end{thm}
	Identifying the exterior of $\mathcal{M}$ with $\mathbb{C} - \overline{\mathbb{D}}$, we can assign each point in $\mathbb{C} - \mathcal{M}$ an angle. 
	More precisely, for any $c \in \mathbb{C} - \mathcal{M}$ the \emph{external angle} of $c$ is $\arg(\Psi^{-1}(c))/2\pi \in [0,1]$. 
    Here, angles are normalized to be between 0 and 1 for convenience and we identify $0$ and $1$ so that the set of external angles form a circle of unit circumference.
	We also call the ray $\Psi^{-1}\big(\{re^{i\alpha}~|~ r \leq 1\}\big)$ as an \emph{external ray} $\mathcal{R}_\alpha$.
    The external angles and external rays here are also called \emph{parameter angles} and \emph{parameter rays} respectively since they are on the parameter plane. 
    We will define external angles and external rays similarly in the complement of Julia sets, and they are called \emph{dynamic angles} and \emph{dynamic rays} respectively since they live in the dynamic plane. 
	We denote $P_\alpha := z^2 + c$, where the external angle of $c \in \mathcal{S}_2$ is $\alpha$, and $J_\alpha$ as its Julia set.
	
	Again by Douady and Hubbard, it is well known that for every parabolic point $c \in \mathcal{M}$ there is a pair of parameter rays $\mathcal{R}_\alpha, \mathcal{R}_{\overline{\alpha}}$ that land at $c$ and vise versa.
	Furthermore such $\alpha$ and $\overline{\alpha}$ are periodic under angle doubling, hence a rational with odd denominators.
	Such pair of angles are called \emph{companion angles}.
	
	\subsection{External angle and Kneading sequences}~
	
	For the polynomial case, there is a nice coordinate change done by B\"{o}ttcher in \cite{boettcher1904principal}.
	\begin{thm}[B\"{o}ttcher coordinates]\label{thm:bottcher}
		Let $P$ be a polynomial of degree $d$.
		Then there exists a pair of neighborhoods $ V\subset U$ of $\infty$ and an analytic mapping $\phi : V \to \mathbb{C}$ which sends $\infty$ to $0$ and satisfying that $\phi'(\infty) = 1$ and $(\phi(z))^d = \phi(P(z))$.
	\end{thm}
	The B\"{o}ttcher coordinates enable us to understand dynamics around infinity, which is semi-conjugation of a monomial map $w\mapsto w^d$.
	Let $G_P := \log |\phi|$, a Green's function.
	\begin{lem}[Properties of Green's function]\label{lem:property_Green_ftn}~
		
		\begin{enumerate}
			\item $G_P : \mathbb{C} \to \mathbb{R}_{\geq 0}$.
			\item The level set of $G_P$ (except at $0$) forms a horizontal foliation of $\mathbb{C} - K_P$.
			\item $K_P = G_P^{-1}(0)$.
			\item Critical points of $G_P$ consists of critical points of $P$ and their preimages under $P$.
		\end{enumerate}
	\end{lem}
	We omit the proof here.
	With the lemma \ref{lem:property_Green_ftn}, we can identify a small neighborhood of infinity as a disc, and hence draw a ray of fixed angle.
	Such ray is called \emph{dynamic ray} of \emph{dynamic angle} $\theta$, and it is orthogonal to every level set of $G_P$.	We also denoted it as $\mathcal{R}_\theta$ by abusing the notation.	In other words, dynamic rays form the vertical foliation orthogonal to the level sets which form the horizontal foliation. 
	
	For periodic $\alpha$ under angle doubling, every dynamic ray of $\mathbb{C} - J_\alpha$ lands at some points in $J_\alpha$ or meets (pre-)critical points and bifurcates.
	Note that $0$ is the only critical point for $P_\alpha$, which is a preimage of the critical value and its dynamic angle is $\alpha$.
	Hence, by B\"{o}ttcher coordinate sense, two dynamic rays of angles $\frac{\alpha}{2}, \frac{\alpha+1}{2}$ land at $0$.
	
	Now, choose any point $p \in J_\alpha$ and let $\theta$ be its dynamic angle. 
	Then the kneading sequence of $p$ (or equivalently, $\theta$) is defined as follows.
	\begin{defn}[Kneading sequence]\label{defn:kneading_seq}
		Suppose the above.
		Divide $S^1$ into $2$ pieces, cut at $\frac{\alpha}{2}$ and $\frac{\alpha+1}{2}$.
		Let $A_\alpha$ be an open arc $\{ \theta ~|~ \frac{\alpha}{2} < \theta < \frac{\alpha+1}{2} \}$ and $B_\alpha$ be the other open arc.
		Then the kneading sequence $I^\alpha(\theta) := x_0x_1x_2\cdots$ is a $0$-$1$ sequence possibly with some $*$'s whose entries are determined by
		\begin{align*}
			x_i = \begin{cases}
				~0 & \text{ if } 2^{i}\theta \in A_\alpha\\ 
				~1 & \text{ if } 2^{i}\theta \in B_\alpha\\ 
				~* & \text{ if } 2^{i}\theta = \dfrac{\alpha}{2} \text{ or } \dfrac{\alpha + 1}{2}
			\end{cases}
		\end{align*}
	\end{defn}
	Hence we construct a map $J_P \to \Sigma_2$, by sending $p\in J_P$ to its itinerary sequence.
	We end up this section with introducing some notations here.
	\begin{itemize}
		\item A $0$-$1$ word $w$ is a finite sequence $w := w_0w_1\cdots w_k$, where $w_i \in \{0, 1\}$.
		Similarly, a $0$-$1$ sequence is an infinite sequence whose entries are all $0$ or $1$.
		\item $\{0, 1\}^{\mathbb{N}}$ is a union of all $0$-$1$ sequences of infinite length.
		\item A $0$-$1$ word $w$ is in $\{0, 1\}^n$ if and only if the length $|w|$ of $w$ is $n$.
		\item $\{0, 1\}^* := \bigcup_{n\in \mathbb{N}}$ $\{0, 1\}^n$ is the collection of every finite $0$-$1$ word $w$. 
		\item For $w\in \{0,1\}^*$, $\overline{w} \in \{0,1\}^\mathbb{N}$ is an infinite sequence with repeating $w$'s.
	\end{itemize}
	These notations will be used in forthcoming sections, which discuss lamination models.
	
	\subsection{Renormalization and Tuning}
	In \cite{douady1986algorithms}, Douady proposed that there is a homeomorphism of Mandelbrot set $\mathcal{M}$ into itself.
	More precisely, let $M_0$ be the main cardioid component of $\mathcal{M}$ and choose any hyperbolic component $M$.
	Then there is a topological embedding $\psi_M : \mathcal{M} \to \mathcal{M}$ satisfying the followings:
	\begin{enumerate}
		\item $\psi_M(0) = c_M$, where $c_M$ is the center of $M$.
		\item $\psi_M(M_0) = M$.
		In other words, $\psi_M$ sends the main cardioid to the hyperbolic compontent $M$ we chose.
		\item $\partial \psi_M(\mathcal{M}) \subset \partial \mathcal{M}$.
	\end{enumerate}
	This is called a \emph{(simple) renormalization} or \emph{tuning}.
	For any $c \in \mathcal{M}$, $\psi_M(c)$ is called \emph{$M$ tuned by $c$} and denoted $M \perp c$.
	
	The tuning affects their Julia sets also.
	Roughly speaking, the filled Julia set $K_{M\perp c}$ is obtained by first drawing $K_{c_M}$ and replacing its every bounded Fatou component into a copy of $K_c$.
	See the figure \ref{fig:rabbit}, \ref{fig:rabbit_tuned_by_basilica} and \ref{fig:mandelbrot_and_basilica}.
	All Julia sets in figures correspond to the center of their hyperbolic components.
	
	\begin{figure}[h]
		\centering
		\begin{subfigure}{0.45\textwidth}
			\centering
			\includegraphics[width=.9\linewidth]{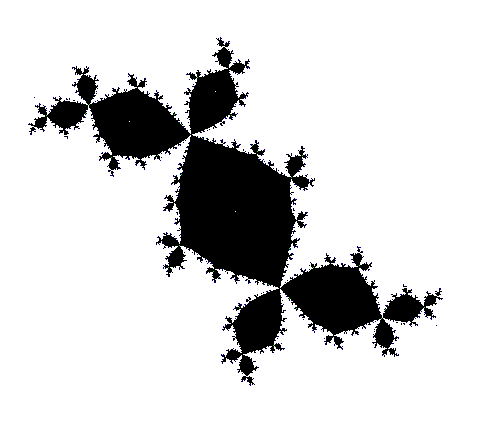}
			\caption{Douady's Rabbit. The parameter angles of the hyperbolic component are $\frac{1}{7}$ and $\frac{2}{7}$.}
			\label{fig:rabbit}
		\end{subfigure}
		\begin{subfigure}{0.45\textwidth}
			\centering
			\includegraphics[width=.9\linewidth]{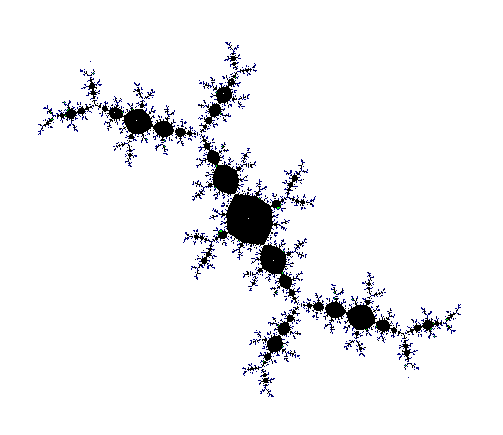}
			\caption{Rabbit tuned by basilica. The parameter angles are $\frac{10}{63}, \frac{17}{63}$}
			\label{fig:rabbit_tuned_by_basilica}
		\end{subfigure}
		\caption{Figures of rabbit and tuned rabbit. Note that the Fatou component of rabbit alters to the basilica.}
	\end{figure}
	\begin{figure}[h]
		\centering
		\includegraphics[width=.8\textwidth]{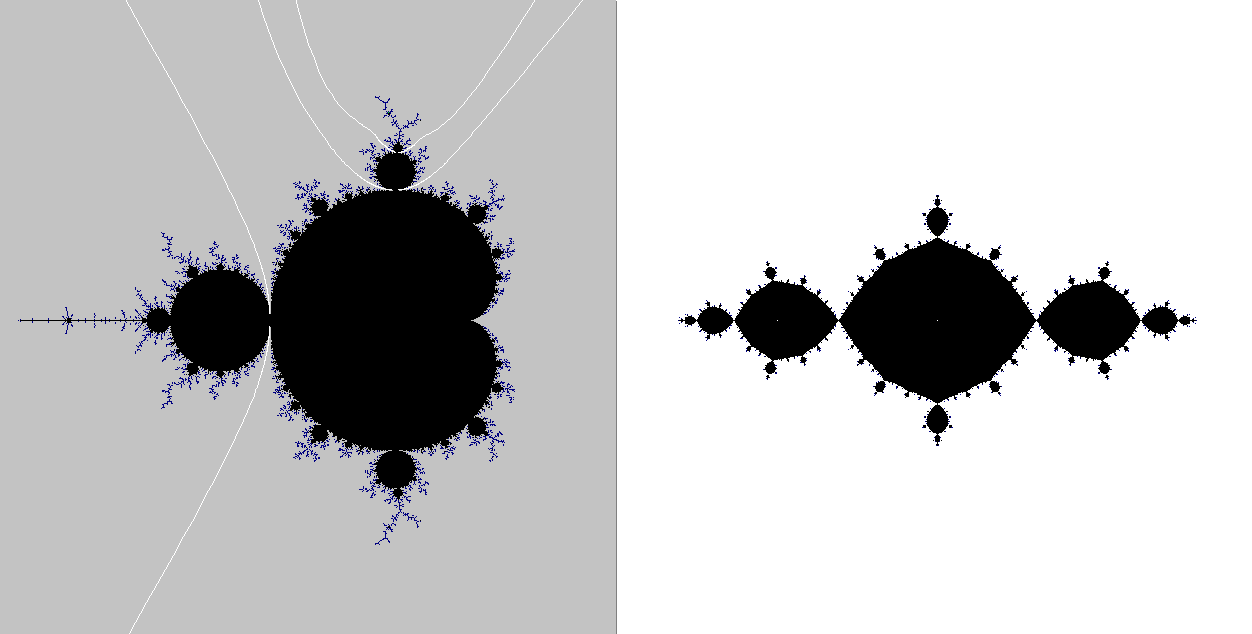}
		\caption{The white lines are the parameter rays, landing on hyperbolic components corresponds to basilica, rabbit, and rabbit tuned by basilica.
		The right figure is a basilica.}
		\label{fig:mandelbrot_and_basilica}
	\end{figure}
	
	There are two kinds of renormalizabilty, one is simple and the other is crossed.
	Instead of saying the details, we refer \cite{mcmullen1994complex} and \cite{mcmullen1996renormalization} for further discussions.
	
	Choose $M$ be any hyperbolic component.
	Suppose $\alpha < \overline{\alpha}$ be a pair of parameter angles of the root of $M$.
	In the same paper, Douady proved the following formula.
	\begin{lem}[Douady's angle tuning]\label{lem:douady_angle_tuning_formula}
		Suppose the above situation.
		Note that $\alpha$ and $\overline{\alpha}$ are both periodic under angle doubling of same period and hence in base $2$ expansion, they have a repeating word $w_0, w_1 \in \{0,1\}^n$ of same length, respectively.
		\emph{i.e.,}
		\[
			\alpha = 0.w_0 w_0 \cdots. \qquad \text{ and } \qquad \overline{\alpha} = 0.w_1 w_1 \cdots.
		\]
		Let $x \in \partial\mathcal{M}$ and there is a parameter ray of angle $\theta$ lands at $x$.
		Suppose the binary expansion of $\theta$ is
		\[
			\theta = 0.t_1t_2t_3 \cdots.
		\]
		Then there is a parameter ray of angle $M\perp \theta$ which lands at $M\perp x$ and,
		\[
			M\perp \theta = 0. w_{t_1}w_{t_2}w_{t_3} \cdots.
		\]
	\end{lem}
	Note that the formula preserves the order.
	In other words, $a<b \Rightarrow (M\perp a) < (M \perp b)$.
	
	There are some remarks about the formula.
	First of all, sometimes binary expansion is not $1$-$1$.
	Especially in $S^1$, we identify $0 = 1$ and hence $0$ attains two different binary expansions $0.000\cdots.$ and $0.111\cdots.$
	After Douady's formula, those two expansions correspond to the pair of parameter angles of the root of $M$.
	
	\section{Characteristic arcs} \label{sec:3-char_arcs}
	
	In this section, we identify $S^1$ as the set of real numbers quotient by $\mathbb{Z}$ and choose a real number in $[0, 1)$ as a representative for each point in $S^1$. There is a canonical cyclic order on $S^1$ which is inherited from the natural total order on $\mathbb{R}$, under the map $x \mapsto e^{2\pi i x}$.
 
	Let's begin with some notations. For $x, y\in S^1$ we denote $(x, y)$ as an \emph{arc} of $S^1$ which is a set of $z$ satisfying that $[x,z,y]$ is positively ordered (with respect to the canonical cyclic order).
	
	We denote $xy$ as a \emph{chord} of $S^1$ whose endpoints are $x, y\in S^1$.	
    Unless specified, we do not care the order of two points. 
    Usually we will draw it like a curve inside a circle, just like a hyperbolic geodesic in the Poincar\'{e} disk.	
    When $x = y$ the chord $xy$ is called degenerated.
	
	The length of arc $l(x,y)$ is measured by the usual metric in $S^1$.
	
	We call a point $p$ lies behind the chord $xy$ if $p$ belongs to the interior of one of $(x,y), (y,x)$ which has length smaller than $1/2$. (the case $xy$ is a diameter is excluded here and this is not necessary in the rest of the paper.) 
	Also, $xy$ is nested by $zw$ (denoted as $xy \subset zw$) if every point $p$ lie behind $xy$ also lie behind $zw$.
	See the figure \ref{fig:arcs_and_chords} also.
	
	\begin{figure}
		\centering
		\includegraphics[width=.5\textwidth]{"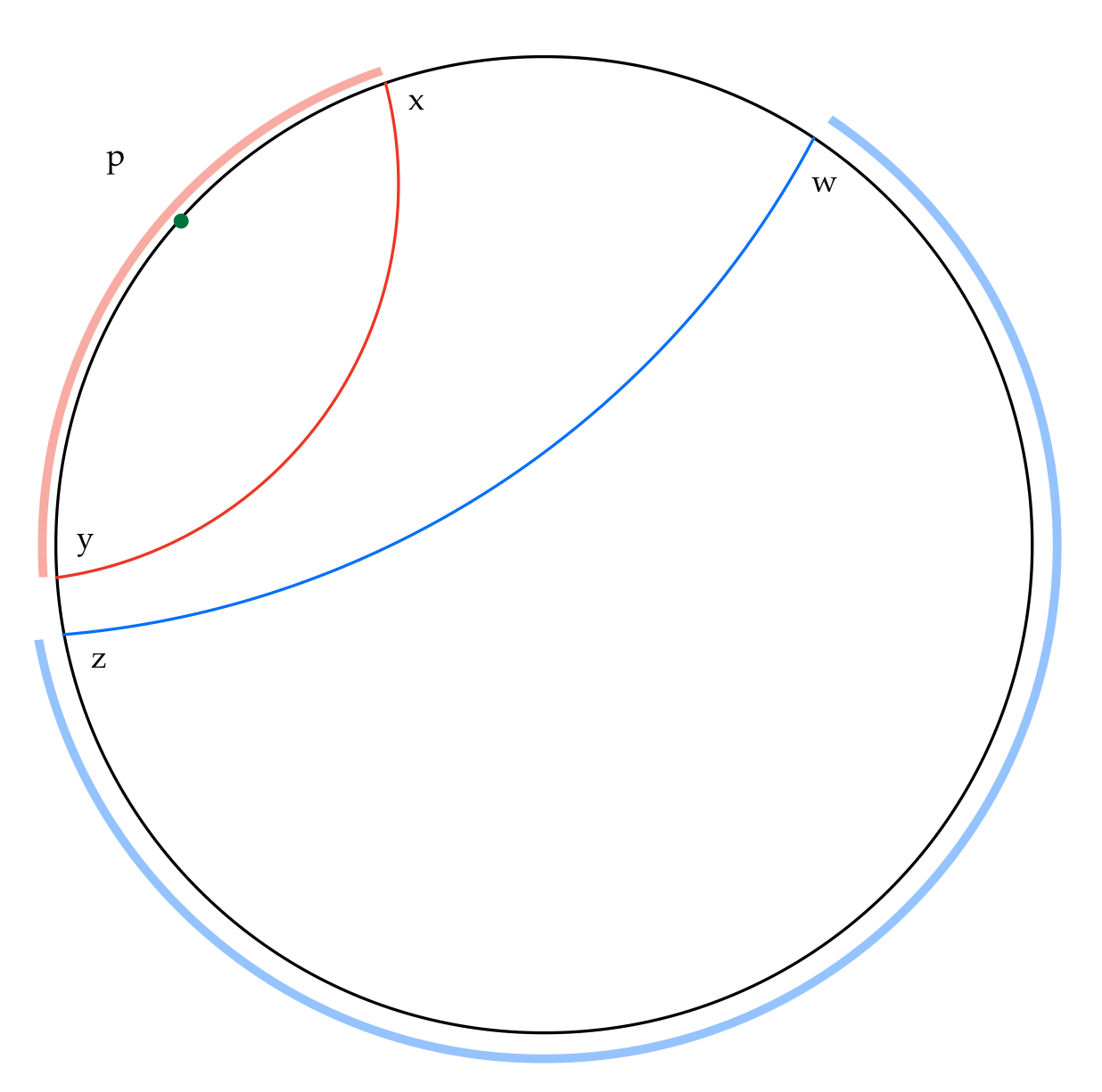"}
		\caption{Here $xy$ and $zw$ are chords. The point $p$ lies behind $xy$ and the chord $xy$ lies behind $zw$. Note that two arcs $(x, y)$ and $(z,w)$ is disjoint.}
		\label{fig:arcs_and_chords}
	\end{figure}
		
	\subsection{Orbit portrait}
	For periodic nonzero $p \in S^1$, in \cite{milnor2000periodic} Milnor introduced a combinatorial model to understand the periodic points in the dynamic plane, called \emph{Orbit portrait}.
	
	\begin{defn}[Orbit portrait]\label{defn:orbit_portrait}
		Let $O_j$ be the set of angles.
		Then the collection $\mathcal{O} :=\{O_1, \cdots, O_p\}$ is called \emph{orbit portrait} if 
		\begin{enumerate}
			\item Each $O_j$ is finite and $|O_i| = |O_j|$ for all $i\neq j$.
			\item For each $j$, the angle doubling map $\theta \mapsto 2\theta \mod \mathbb{Z}$ is a bijective map between $O_j$ to $O_{j+1}$, which preserving the cyclic order of $\mathbb{R}/\mathbb{Z}$.
			\item For any $a\in \cup_{j = 1}^p O_j$, $a$ is periodic under angle doubling of the same period $rp$, a multiple of $p$.
			\item $O_j$'s are pairwise unlinked.
			More precisely, for each $i\neq j$, the line/polygon $\mathbf{P}_i$ inscribed in the unit circle $S_1$ whose vertices are points in $O_i$ has no intersections with $\mathbf{P}_j$.
		\end{enumerate}
	\end{defn}
	$p$ is called orbit period, and $rp$ is called ray period or angle period.
	There are examples of orbit portraits in figure \ref{fig:Orbit_portrait_ex}.
	
	\begin{figure}[h]
		\centering
        \includegraphics[width=.9\linewidth]{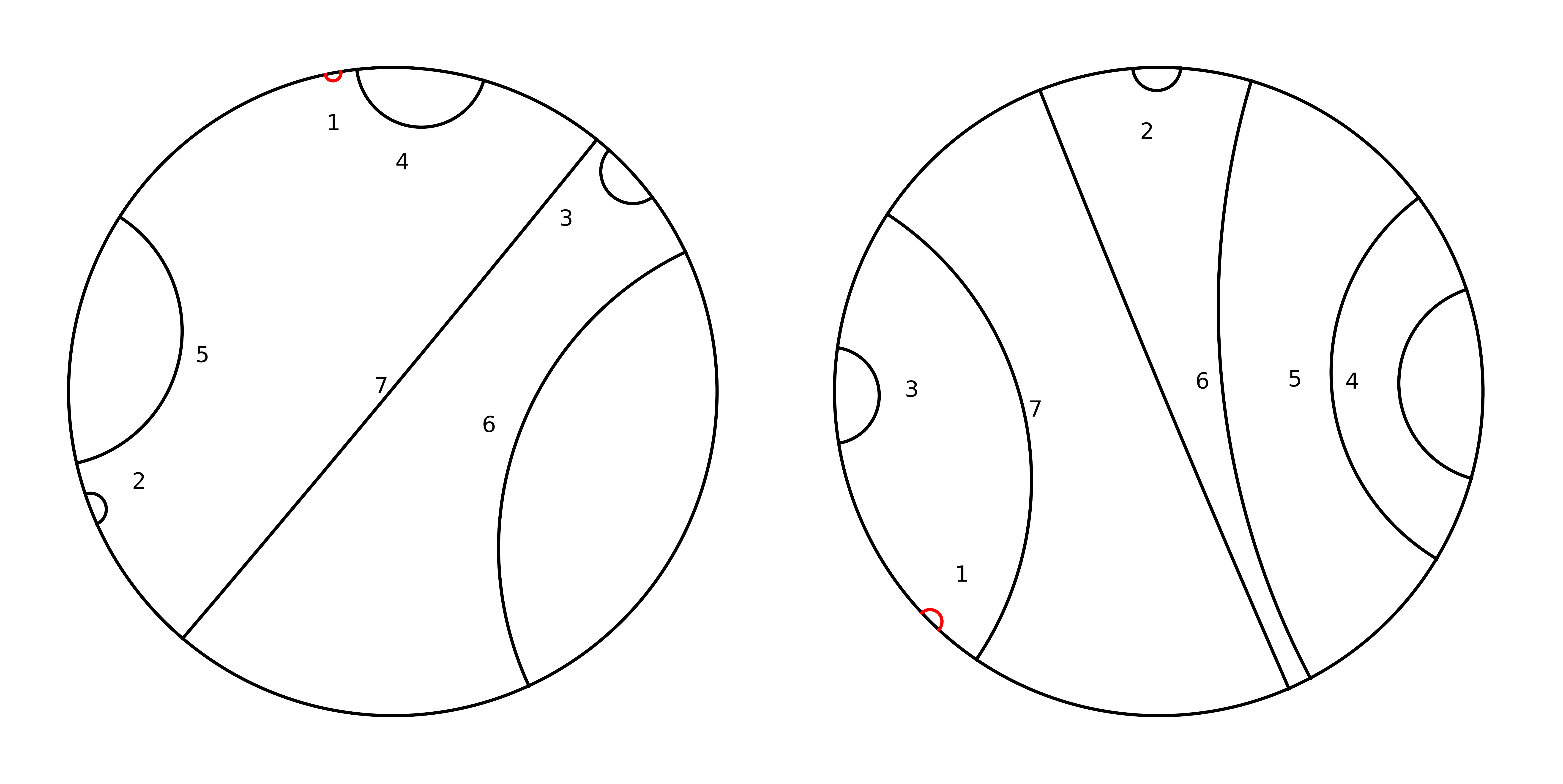}
		\caption{Examples of characteristic arcs.
        Characteristic arcs for left and right orbit portraits are $\left( \frac{35}{127}, \frac{36}{127} \right), \left( \frac{158}{255}, \frac{161}{255} \right)$, respectively.
        We denote by the number for each $i$th image of $A_1$.
        Note that the left one is narrow and the right one is not narrow.
		Hence the characteristic arc on the right is nested by the $7$th orbit.}
		\label{fig:Orbit_portrait_ex}
	\end{figure}

	The last statement implies the orbit portrait yields a lamination when we draw a line/polygon of endpoints of each $O_j$'s.
	An orbit portrait captures the dynamics of any repelling periodic point of given Julia set. In particular, we have the following.
	\begin{thm}[\cite{hubbard2016teichmuller}, Theorem 10.5.13]\label{lem:char_arc_lands_at_root} 
		Suppose $I = (\alpha, \overline{\alpha})$ is a characteristic arc of period $n$ and $M$ is a corresponding hyperbolic component.
		Let $c_M$ be the center of $M$ so that $p(z) = z^2 + c_M$ is postcritically finite.
		Then the dynamic rays of angle $\alpha$ and $\overline{\alpha}$ land at the root of Fatou component which contains $c_M$.
	\end{thm}
	In other words, if there is a polygon gap in the orbit portrait, then some Fatou components which contain critical orbit meet at one point.
	
	Complementary arcs for each $A_j \in \mathcal{O}$ are subintervals obtained by cutting at each point in $A_j$.
	By the lemma \ref{defn:orbit_portrait}, the angle doubling map $h$ still be a bijection between the set of complementary arcs of $A_j, A_{j+1}$, except one.
	More specifically, $2$ complementary arcs of length $>1/2$ cannot exist in one set.
	Suppose $L_j$ be the complementary arc of length $>1/2$ if exists.
	Then except $L_j$, $h$ sends the others diffeomorphically, and $L_j$ covers whole circle of length $1$.
	The overlapped region should be one of the complementary arcs of $A_{j+1}$, since the sum of length should be $2$ and there are only $1$ piece left after taking $h$.
	The below lemma elaborate the above.
	\begin{lem}[\cite{milnor2000periodic}, Lemma 2.5]\label{lem:critical_arc}
		The set of complementary arcs of $A_j$ sends diffeomorphically to the one of $A_{j+1}$ (modulo its period) except one.
		For each $j$, $L_j$ exists and its image under $h$ covers the whole circle, and there is a unique complementary arc for $A_{j+1}$ which is covered only once by the image.
	\end{lem}
	A unique minimum length among all complementary arcs is attained, and any complementary arc realizing the minimum is called \emph{characteristic arc} of $\mathcal{O}$.

    Abusing the notation above, we enumerate the orbit under angle doubling of the characteristic arc $I$ as follows.
	Set $I := A_1$.
	If $l(A_i) < 1/2$, then $A_{i+1}$ is an image of $A_i$ under angle doubling.
	Otherwise, we set $A_{i+1}$ as the unique complementary arc obtained from lemma \ref{lem:critical_arc}.
	
	There is a $1$-$1$ correspondence between the collection of characteristic arc and a bifurcation locus.
	Each point in the bifurcation locus, there is a unique hyperbolic component which admits the point as its root.
	\emph{i.e.,} we can identify characteristic arc to its hyperbolic component.
	More precisely, if $c$ is a point in the bifurcation locus, then $z^2 + c$ has a parabolic cycle of points.
	Draw all rays which land at points in the parabolic cycle, then it forms an orbit portrait.
	Milnor proved that every orbit portrait can be realized as a parabolic cycle of certain quadratic polynomials.
	
	Moreover, if one characteristic arc $I$ is nested by the other characteristic arc $J$, then those orbit portrait are unlinked to each other.
	In a dynamical sense, if $c$ belongs to a hyperbolic component corresponding to $I$, then the Julia set of $z^2+c$ has a repelling periodic point whose landing rays form an orbit portrait of $J$.
	This is called \emph{Orbit forcing}.
	
	One may collect all the characteristic arcs and draw them in the unit disk.
	These arcs form a lamination which is called \emph{quadratic minor lamination (or simply QML)}, as first proposed by Thurston. 
	We provide the picture of QML in figure \ref{fig:QML}.
    Characteristic arcs have properties which enable us to get all arcs in an algorithmic way. 
	\begin{lem}\label{lem:Lavarus_lemma}
		Every leaf in QML satisfies the followings.
		\begin{enumerate}
			\item Every periodic angle is an endpoint of non-degenerate leaf of QML, whose the other endpoint has same period.
			\item Every leaf of QML is the limit of leaves with periodic endpoints.
			\item Suppose two leaves in QML have equal period $p$ and one is nested by the other.
			Then there exists a leaf $L$ so that it has period $<p$ and it separates two leaves. 
		\end{enumerate}
	\end{lem}
	The first $2$ properties implies that to generate QML it suffices to find all periodic characteristic arcs.
	Lavarus proposed in \cite{lavaurs1986description} an algorithm to get all characteristic arcs. 
	\begin{lem}[Lavarus algorithm, \cite{lavaurs1986description}]\label{lem:Lavarus_algorithm}
		Start with a leaf $(1/3, 2/3)$.
		Suppose $C_n$ is a collection of all characteristic arcs whose period less than $n$.
		For $p = n+1$, consider all points of $k/(2^{n+1}-1), 1\leq k \leq 2^{n+1}-2$.
		Connect them pairwise starting with $1/(2^{n+1}-1)$ by the rule as follows.
		\begin{enumerate}
			\item Choose the point in the order of increasing angles.
			\item The other endpoint is the lowest angle among all angles which is not linked with any other leaf in the collection $C_n$.
		\end{enumerate} 
		If every point of period $(n+1)$ finds its pair, increase $p$ by $1$.
		Iteratively, the algorithm collects all characteristic arcs.
	\end{lem}
	After drawing all characteristic arc in the unit circle, taking the closure (add all accumulated arcs) becomes a famous lamination, called QML (quadratic minor lamination), which is first introduced by Thurston.

	\begin{figure}
		\centering
		\includegraphics[width=.8\textwidth]{"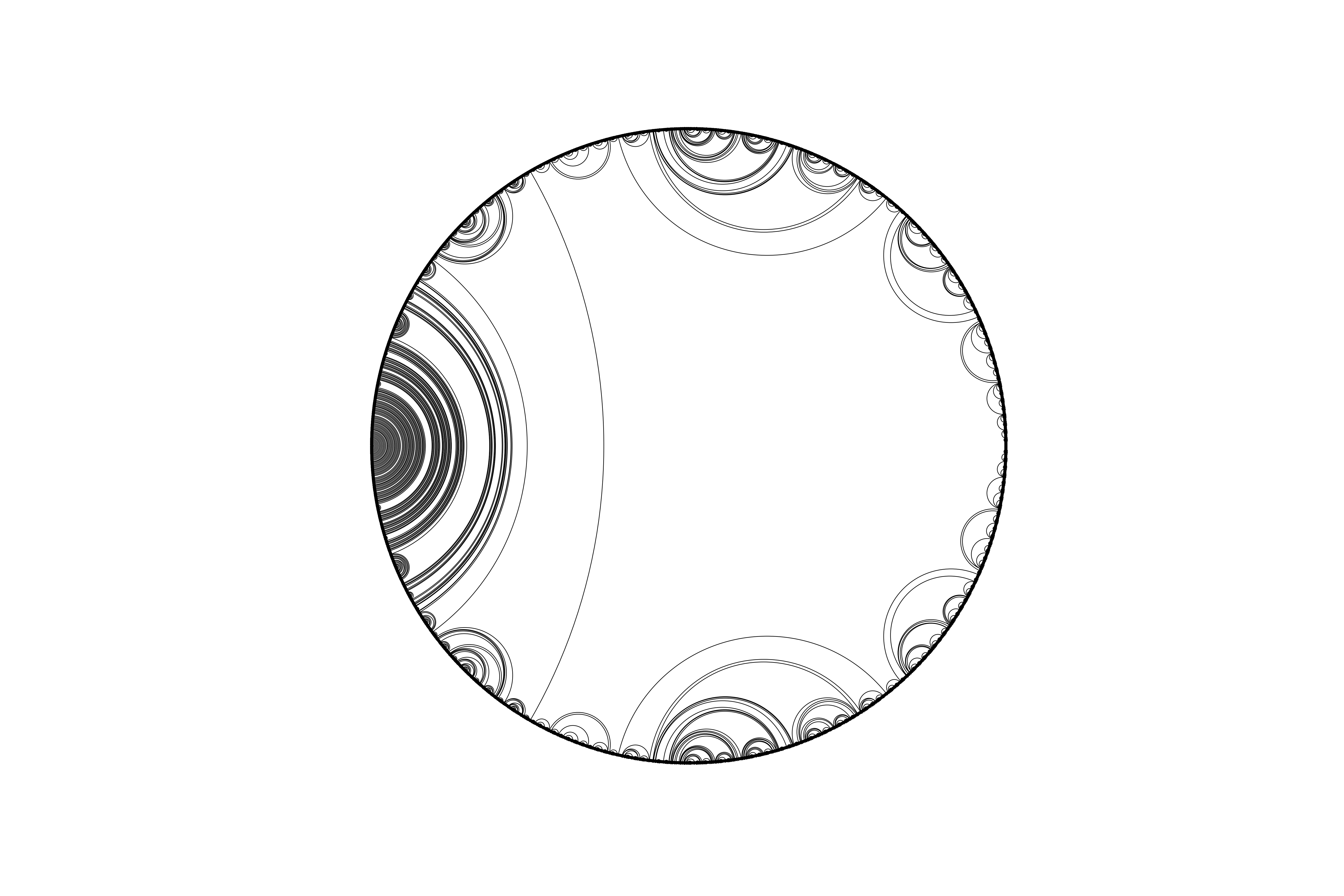"}
		\caption{QML, in this figure we only draw arcs up to period $10$.}
		\label{fig:QML}
	\end{figure}
 	
	Suppose $z_1$ is a parabolic fixed point of $P$ and $\{z_1, \cdots, z_p\}$ be its periodic orbit under $P$.
	There are two types of parabolic points, one is \emph{primitive} and the other one is \emph{satellite}.
	\begin{defn}\label{defn:primitive_satellite}
		Suppose $\alpha \in S^1$ be nonzero periodic and $\overline{\alpha}$ is its companion angle.
		$\alpha$ is called \emph{satellite} if the orbit of $\alpha$ coincides to that of $\overline{\alpha}$.
		Otherwise, $\alpha$ is called \emph{primitive}.
	\end{defn}
	
	We briefly remark some properties of primitive/satellite points.
	\begin{rmk}[See Lemma 2.7 in \cite{milnor2000periodic}]~
		
		\begin{enumerate}
			\item If $\alpha$ is primitive, there are no polygon gap in the orbit portrait.
			With respect to the dynamical plane and rays, there are at most $2$ rays land on each point in the periodic orbit.
			\item If $\alpha$ is satellite, then there is a $r$-gon gap in the orbit portrait. 
			\emph{i.e.,} exactly $r$ rays lands at each point of the periodic orbit.
		\end{enumerate}
	\end{rmk}
	
	\subsection{Narrow arcs}
	Another criterion to classify characteristic arc is by its length.
	\begin{defn}[Narrow arc, narrow component, \cite{schleicher2017internal}]
		Suppose $I$ is a characteristic arc of angle period $n$.
		Then $I$ is called \emph{narrow arc} if $ l(I) = \frac{1}{2^n - 1}$.
		The hyperbolic component corresponds to $I$ is called \emph{narrow component}.
	\end{defn}
	By definition and the lemma \ref{lem:Lavarus_lemma}, if a characteristic arc $I$ is not narrow, then there is a shorter arc lie behind $I$ whose period is strictly less than $n$.
	
	We prove the following property.
	\begin{lem}\label{lem:not_narrow_has_unique_longest_narrow}
		Suppose $I$ is a characteristic arc of period $n$
		If $I$ is not narrow and $I$ satisfies $\frac{1}{2^k-1} < l(I) < \frac{1}{2^{k-1}-1}$, then there is a narrow arc $J$ nested by $I$, whose period is $k$.
		Moreover $k$ is strictly less than $n$.
	\end{lem}
	\begin{proof}
		Since $\frac{1}{2^k-1} < l(I) < \frac{1}{2^{k-1}-1}$, there are $2$ or $3$ points of period $k$ lie behind $I$.
		By lemma \ref{lem:Lavarus_lemma}, it must be $2$ points and those have to be endpoints of some characteristic arc $J$, because of the unlinked condition of QML.
		Then $l(J) = \frac{1}{2^k-1}$.
		Hence $J$ is narrow of period $k$.
		If $k = n$, $l(I) < \frac{1}{2^{n-1}-1}$ implies $l(I) = \frac{1}{2^n-1}$.
		Hence a contradiction.
	\end{proof}
	Note that such arc $J$ is the longest narrow arc among all narrow arcs which lie behind $I$.
	Using the lemma \ref{lem:not_narrow_has_unique_longest_narrow} we prove the equivalence definition of narrow condition of $I$.
	\begin{prop}\label{prop:narrow = no nesting}
		Suppose $I = A_1$ is a characteristic arc of period $n$.
		Then $I$ is narrow if and only if $I$ is not nested by  $A_i$ for all $2 \leq i < n$.  
	\end{prop}
	\begin{proof}
		We denote $I$ as $A_1$.
		Suppose to the contrary that $A_1$ is narrow and it is nested by $A_k$ for some $2\leq k \leq n-1$.
		$l(A_1) = \frac{1}{2^n-1}$ implies $l(A_{i}) = \frac{2^{i-1}}{2^n-1}$.
		Thus $A_n$ is the only arc of length $>1/2$ and $A_i \to A_{i+1}$ is just an angle doubling for all $i \leq n-1$.
	    Let $x<y$ be the endpoint of $A_1$.
	 	The preimage of $x$ and $y$ under angle doubling are $\frac{x}{2}, \frac{x+1}{2}$ and $\frac{y}{2}, \frac{y+1}{2}$	, respectively.
	 	Since $A_{k-1} \to A_k$ is an angle doubling, endpoints of $A_{k-1}$ must contain either $\frac{x}{2}, \frac{y}{2}$ or $\frac{x+1}{2}, \frac{y+1}{2}$.
	 	But $A_n$ is one of the long side of quadrilateral whose vertices are preimages, $A_n$ and $A_{k-1}$ intersects.
	 	Therefore we conclude that if $A_1$ is narrow then it is never nested or nested only by $A_n$.
	 	
	 	Now, suppose $I$ is not narrow and satisfies $\frac{1}{2^m-1} < l(A_1) < \frac{1}{2^{m-1}-1}$.
	 	By lemma \ref{lem:not_narrow_has_unique_longest_narrow}, there is a unique longest narrow arc $B_1$ of period $m < n$.
	 	Denote $B_j$'s as its orbit.
	 	We split it into $2$ cases.
	 	See the figure \ref{fig:narrow_proof} also.	 	
	 	\begin{figure}
	 		\centering
	 		\includegraphics[width = .5\textwidth]{"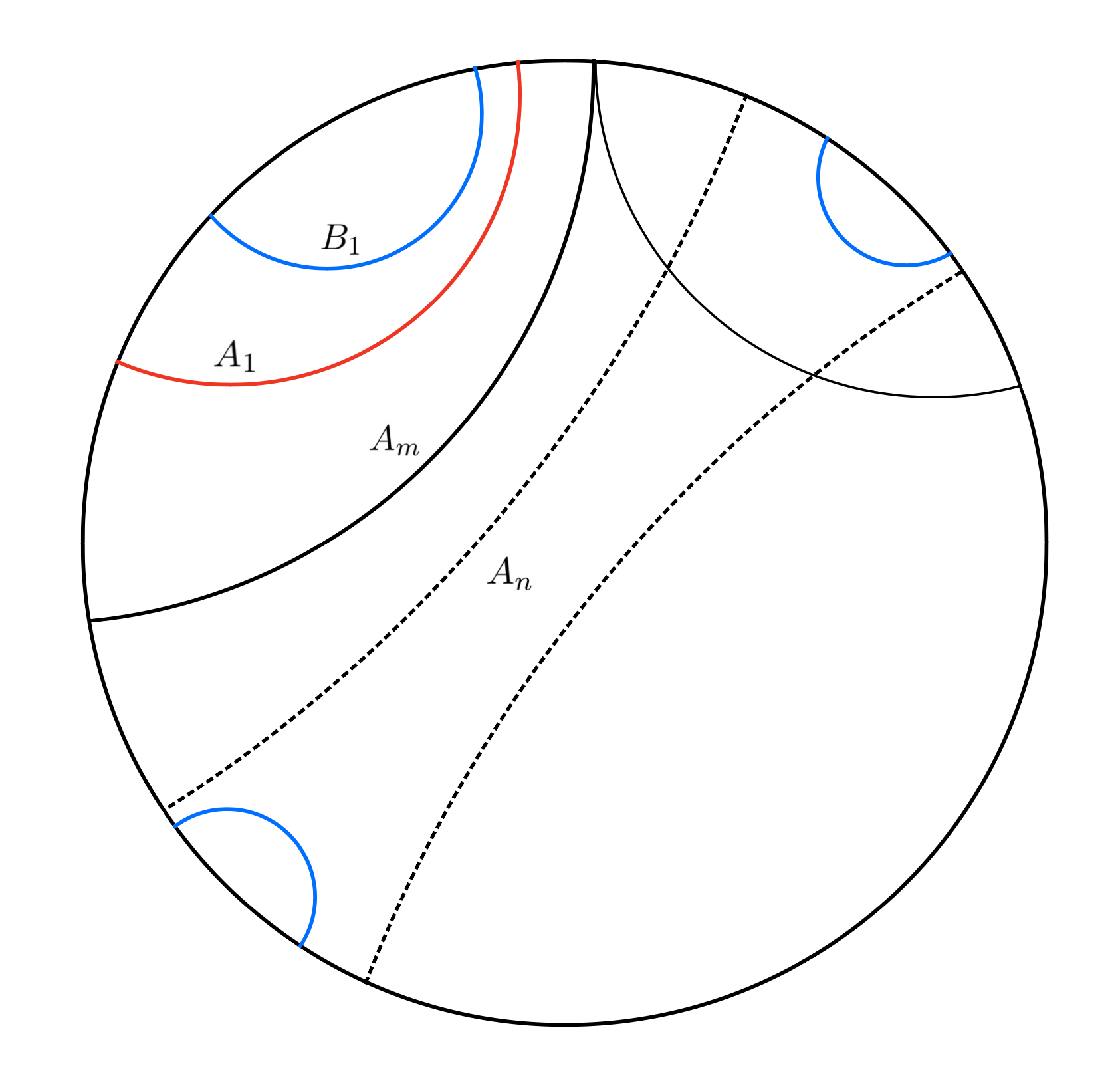"}
	 		\caption{In the middle there are two candidate chords of $A_n$. Then $A_m$ cannot nest the preimage of $A_1$ without intersecting $A_n$.}
	 		\label{fig:narrow_proof}
	 	\end{figure}
	 	
	 	\begin{itemize}
	 		\item $l(A_1) > \frac{1}{2^{m-1}}$.
	 		
	 		Then $\frac{1}{2^{m-1}} < l(A_1) < \frac{1}{2^{m-1}-1}$ gives $\frac{2^{i-1}}{2^{m-1}} < l(A_i) < \frac{2^{i-1}}{2^{m-1}-1}$ for $1\leq i \leq m$ and hence $A_{m-1}$ is the first $A_i$ whose length is longer than $1/2$.
	 		Endpoints of $A_{m-1}$ cannot be in the preimage of $A_1$ unless one of the complementary arc of $A_m$ is shorter than $A_1$, which is impossible.
	 		As $A_1$ nests $B_1$, the preimage of $B_1$ is a (proper) subset of $A_1$.
	 		This implies $A_m$ nests $B_1$.
	 		Since $A_1$ is the shortest among all other $A_j$, $A_m$ cannot be located between $B_1$ and $A_1$.
	 		Therefore $A_m$ nests $A_1$ and $m$ is strictly less than $n$, so we are done.
	 		
	 		\item $l(A_1) < \frac{1}{2^{m-1}}$. 
	 		
	 		Same strategy as above, we get $A_{m+1}$ nests $A_1$.
	 		So we are done if $m+1 < n$.
	 		Suppose $m+1 = n$.
	 		Then again $A_m$ is the first $A_i$ whose length is longer than $1/2$.
	 		Let $S_m$ be $S^1 - A_m$.
	 		If $A_m$ contains $B_1$, then $A_{m-1}$ must contain one of the preimage of $B_1$ since $A_{m-1} \to A_m$ is just an angle doubling map.
	 		$A_m$ is strictly longer than $l(A_1)/2$ and hence contain the preimage of $A_1$ under angle doubling (, not in an orbit), which induces intersection between $A_n$.
	 		This implies $S_m$ contains $B_1$ and therefore $A_m$ must nest $A_1$.
	 	\end{itemize}
 		Same analogy as in the latter case, if we replace $A_m$ to $A_n$, we get $A_n$ must nest $A_1$.
	\end{proof}
	The proposition implies that the characteristic arc $I$ is always nested by some $A_m$ with $m$ less or equal to the period of $I$, and $m = n$ if and only if $I$ is narrow.
	\begin{lem}\label{lem:narrow_starts_with_odd}
		Suppose $I = (\frac{a}{2^n-1}, \frac{a+1}{2^n-1})$.
		Then $a$ must be an odd number.
	\end{lem}
	\begin{proof}
		By assumption, $I = A_1$ is narrow.
		Hence $A_n$ nests $A_1$.
		Consider the preimage of $A_1$ under angle doubling map.
		It consists of $2$ short arcs, and if one draw the polygon whose vertices are the endpoints of such arcs, $A_n$ must be the longer side which nest $A_1$.
		This implies one of the endpoint of $A_n$ is $\frac{a+1}{2(2^n-1)}$.
		Therefore $a$ must be odd unless $\frac{a+1}{2(2^n-1)}$ became strictly preperiodic.
	\end{proof}

	\begin{figure}[h]
		\centering
		\includegraphics[width = .7\textwidth]{"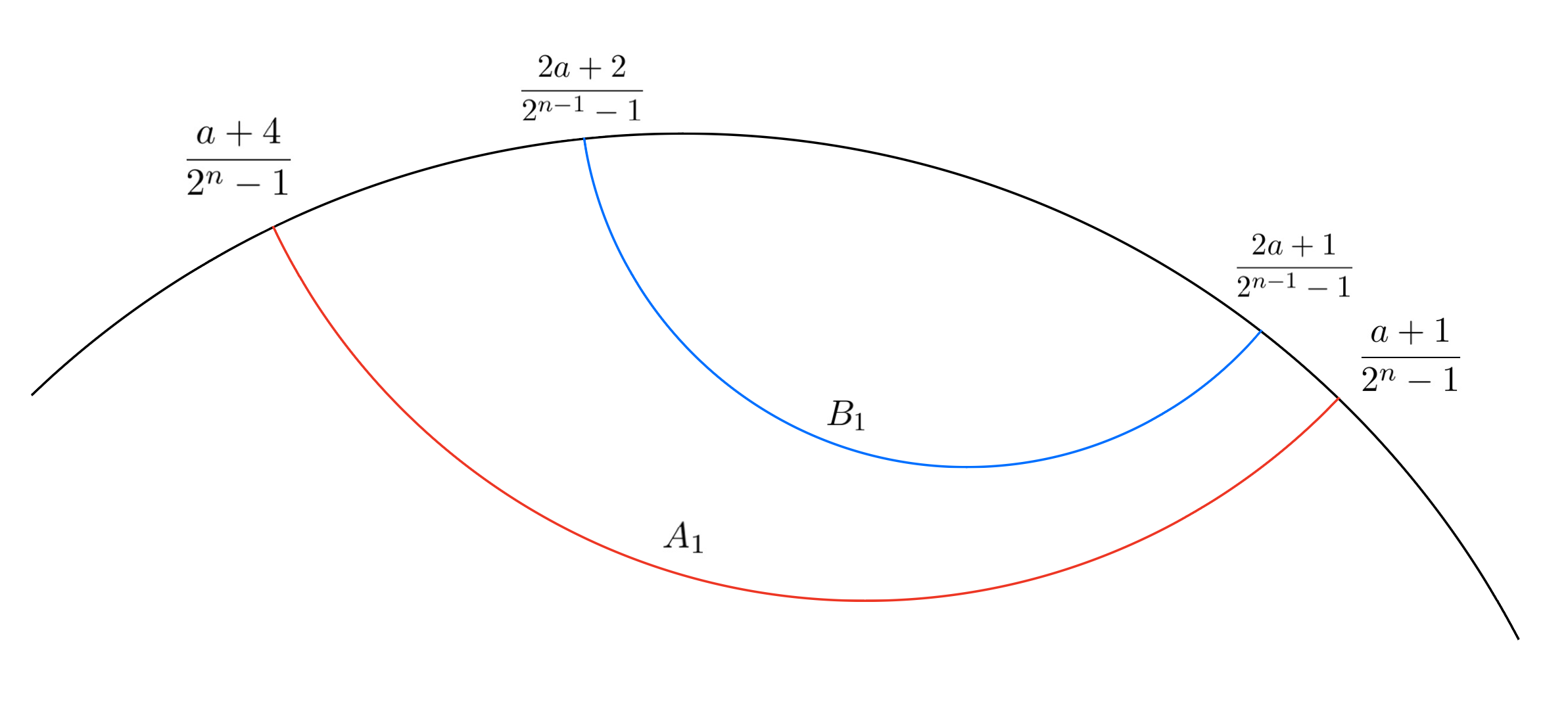"}
		\caption{Examples in the remark.}
		\label{fig:example_of_narrow}
	\end{figure}

	\begin{rmk}
		\begin{enumerate}
			\item The lemma \ref{prop:narrow = no nesting} states that $A_m$ is the first one which nests $A_1$.
			\item In the latter case, only $l(A_1)$ satisfying the condition $\frac{1}{2^{n-1}-1} < l(A_1) < \frac{1}{2^{n-2}-1}$ is $l(A_1) = \frac{3}{2^n-1}$.
			There are only $1$ case for $A_1, B_1$ as below, because of the lemma \ref{lem:narrow_starts_with_odd}.		
			\[
				\frac{a+1}{2^n-1} < \frac{2a+1}{2^{n-1}-1} < \frac{2a+2}{2^{n-1}-1} < \frac{a+4}{2^n-1}
			\]
			These are depicted in figure \ref{fig:example_of_narrow}.
		\end{enumerate}
	\end{rmk}
	\subsection{Simply renormalizable arcs}
	Suppose $\mathcal{P}$ and $\mathcal{Q}$ are hyperbolic components, and $I$ is a characteristic arc corresponds to the hyperbolic component $\mathcal{P}*\mathcal{Q}$.
	We first note that $\mathcal{Q} \mapsto \mathcal{P}*\mathcal{Q}$ gives a tuning from any point $c\in \mathcal{Q}$, and hence $p(z) = z^2+c$ with $c\in\mathcal{P}*\mathcal{Q}$ is renormalizable.
	Choose $c$ a center of $\mathcal{P}*\mathcal{Q}$ and consider the orbit portrait of $I$.
	Since $\mathcal{P}*\mathcal{Q}$ is contained in the wake of $\mathcal{P}$, the orbit portrait $\mathcal{O}_\mathcal{P}$ corresponds to $\mathcal{P}$ is unlinked with the orbit portrait of $I$.
	Furthermore, the orbit portrait $\mathcal{O}_\mathcal{P}$ partition the arcs $I = A_1, A_2, \cdots, A_n$ into same number of $m$ subsets $\{A_1, A_{m+1}, \cdots A_{(k-1)m+1}\}$, $\{A_2, A_{m+2}, \cdots, A_{(k-1)m+2} \}$, 
	$\cdots \{ A_m, A_{2m}, \cdots, A_{km} (= A_n)\}$, where $m$, $k$ and $n = km$ is the period of $\mathcal{Q}$, $\mathcal{P}$, and $\mathcal{P}*\mathcal{Q}$, respectively.
	
	The small Julia set $K(1)$ corresponds to $c_M$.
	The local map $p^{\circ m}$ is hybrid equivalent to some polynomial, whose constant term belongs to $\mathcal{Q}$.
	Hence, locally there is an orbit portrait corresponding to $\mathcal{Q}$ and also for the other $K(i)$'s.
	
	By the Douady's angle tuning formula \ref{lem:douady_angle_tuning_formula}, we can recover each orbit portrait into the whole orbit portrait of $\mathcal{I}$. 
	Also, each small orbit portrait is separated by chords in $\mathcal{O}_\mathcal{P}$.
	
	\emph{i.e.,} If given characteristic arc $I$ is simply renormalizable, then there exists a characteristic arc $J$ which satisfies the followings,
	\begin{itemize}
		\item $I$ is nested by $J$.
		\item Period of $J$ is a integer multiple of period of $I$.
		The orbit portrait of $J$ separates the orbits in $I$ into subsets of same cardinality.
	\end{itemize}
	These condition gives a sufficient condition for renormalizablity.
	\begin{lem}\label{lem:renorm_implies_nesting}
		Suppose $I$ is a renormalizable arc of period $n$.
		Then $I := A_1$ is nested by $A_i$ for some $i$ strictly less than $n$.
		Therefore, $A_1$ cannot be narrow.
	\end{lem}
	\begin{proof}
		With above discussions, the characteristic arc $I$ is in the small partition and form a orbit portrait.
		Let this orbit portrait $\mathcal{O}_{inn}$, and denote each arc as $I = B_1, \cdots, B_k$.
		By proposition \ref{prop:narrow = no nesting}, $I$ is nested by $B_l$ for some $l \leq k$.
		After applying the Douady's angle tuning formula \ref{lem:douady_angle_tuning_formula}, $A_1$ is nested by $A_{(l-1)m+1}$ which is obviously different from $A_n$.
		Here $km = n$.
		Therefore $A_1$ cannot be narrow by the proposition \ref{prop:narrow = no nesting}.
	\end{proof}
	
	\begin{cor}\label{cor:simply_renormalizable_criterion}
		Every narrow arc/component is not simply renormalizable.
		Every prime period arc/component is not simply renormalizable.
		Every satellite arc/component is not narrow unless its root meets in the main cardioid $M_0$.
	\end{cor}
	However the converse does not hold.
	For example, there are non-narrow characteristic arc of prime period.
	It is easy to verify that to be renormalized, the period must have a divisor except $1$ and itself and hence it is not simply renormalized.
	
	Unfortunately, there are many examples which is non-prime, non-narrow and not simply renormalizable.
	We left to the reader to see the table in the appendix \ref{appendix:B-3non}. 
	We found such examples by using the internal address of hyperbolic components, proposed by Schleicher in \cite{schleicher2017internal}.
	In fact, the internal address gives a complete classification of renormalizability, including cross-renormalizability.
	We end the section up with a question.
	\begin{que}
		What else geometric conditions should be required to be non-renormalizable?
	\end{que}
	\begin{figure}[h]
		\centering
		\includegraphics[width = .5\textwidth]{"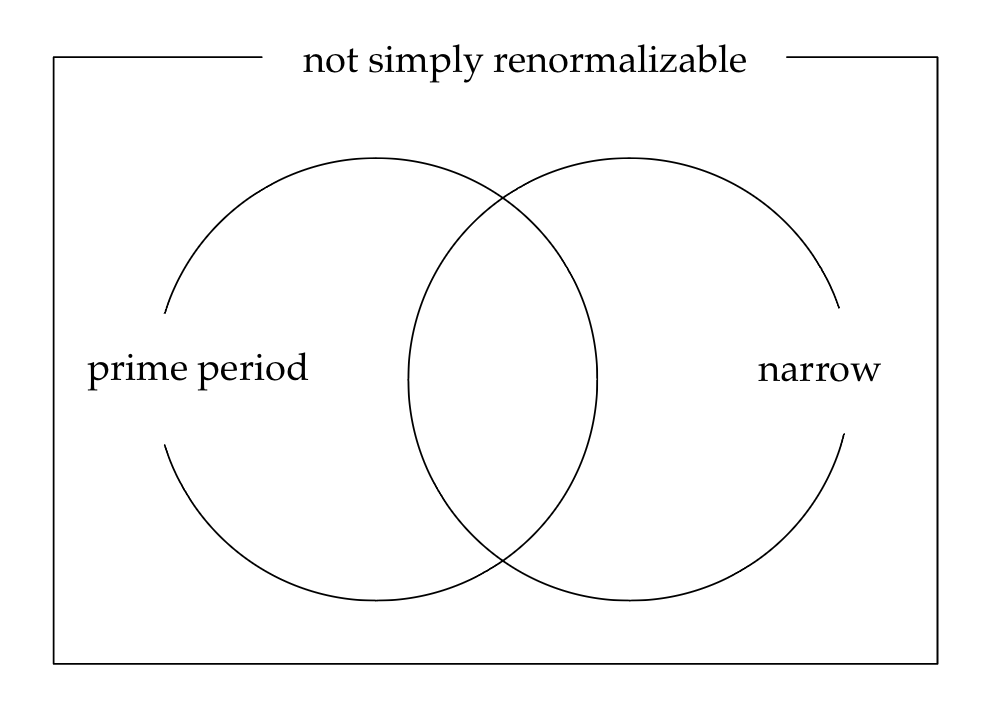"}
		\caption{Inclusion relations between narrow, prime period and not simply renormalizable conditions. Note that there are plenty of not simply renormalizable arcs which are non-narrow, non-prime period. We provide a table in the appendix \ref{appendix:B-3non}.}
		\label{fig:Venn_diagram}
	\end{figure}

	\section{Invariant Laminations} \label{sec:4-inv_lam}
	Most of contents in this section is done by Keller, we refer to the reader \cite{keller2007invariant} for further details.
	Here we endow his notations and theorems.
	\subsection{Lamination model generated by a long chord}
	Let $h$ be an angle doubling map in $S^1$. 
	\emph{i.e.,} $h(x) = 2x \mod 1$.
	Fix $\alpha\in S^1$ and suppose $\alpha \neq 0$ is periodic under $h$.
	In other words, there is the smallest $n\in \mathbb{N}$, a period of $\alpha$, such that $h^n(\alpha) = \alpha$.
	Denote $\PER(\alpha)$ as a period of $\alpha$ and $v^\alpha$ be its repeating word in the kneading sequence $\hat{\alpha} := I^\alpha(\alpha)$.
	Note that $v^\alpha$ always starts with $0$, since $\alpha$ lies on $A_\alpha := \frac{\alpha}{2}\frac{\alpha+1}{2}$.
	Since $\alpha$ is periodic, one of $\frac{\alpha}{2}, \frac{\alpha + 1}{2}$ should be equal to $h^{\PER(\alpha)-1}(\alpha)$.
	Such point is denoted by $\dot{\alpha}$, and the other one is denoted by $\ddot{\alpha}$, which is preperiodic.
	
	Consider the branches of inverse $h^{-1}$ whose domain is $S^1 - \{\alpha\}$.
	Denote one branch as $l^{\alpha}_0$ if the image is $(\frac{\alpha}{2}, \frac{\alpha + 1}{2})$, and the other as $l^{\alpha}_{1}$ if the image is $(\frac{\alpha + 1}{2}, \frac{\alpha}{2})$.
	To extend the domain to the whole $S^1$, define $l^{\alpha, t}_s(\alpha) = \dot{\alpha}$ if $t = s$, otherwise $l^{\alpha, t}_s(\alpha) = \ddot{\alpha}$ for $t,s \in \{0,1\}$.
	
	Graphically, the map $l^{\alpha, t}_s$ determines which endpoint of arc $(\frac{\alpha}{2}, \frac{\alpha + 1}{2})$ or $(\frac{\alpha}{2}, \frac{\alpha + 1}{2})$ is the preimage of $\alpha$, depends on the parity of $s$ and $t$.
	
    \begin{figure}
        \centering
        \includegraphics[width=.9\textwidth]{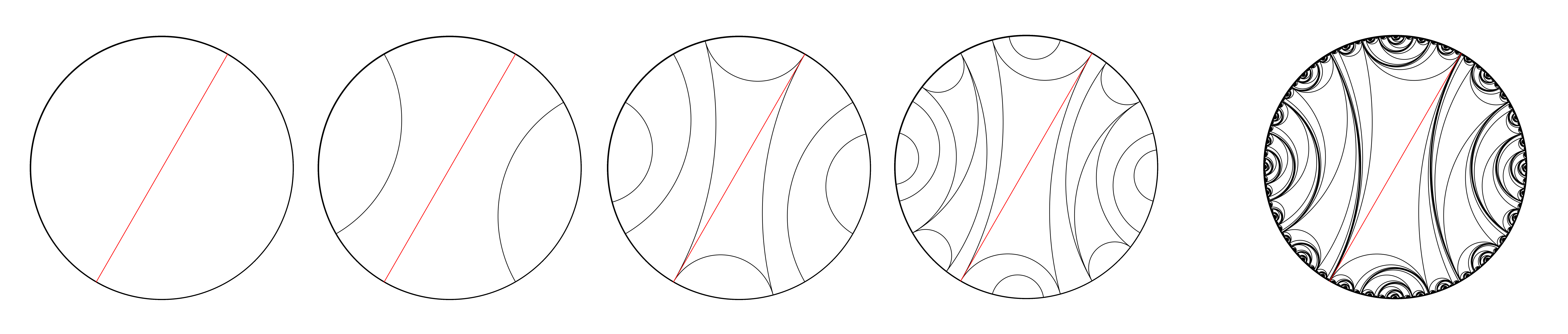}
        \caption{$\alpha = \frac{1}{3}$. Starting with red line $\frac{1}{6}\frac{2}{3}$ preimages are added. Figures depict the step $1$ to step $4$, and the rightmost one depicts $\mathcal{L}_\alpha$.}
        \label{fig:deg2lam_ex}
    \end{figure}
 
	Let $w = w_1w_2\cdots w_n$ be a finite length word with $0, 1$ symbols.
	Iteratively, we can define $l^{\alpha, t}_w := l^{\alpha, t}_{w_1} \circ \cdots l^{\alpha, t}_{w_n}$.
	If the input is in $S^1 - \{\alpha, h(\alpha), \cdots, h^{n-1}(\alpha)\}$, then the superscript $t$ is irrevalent.
	
	We first note that the branched inverse is considered as a prefix attaching map on the $\{0, 1\}^\mathbb{N}$.
	\begin{lem}
		Let $p(z) = z^2+c$ that the external angle of $c$ is $\alpha$.
		Suppose $x\in J_p$ attains its iternary sequence $x_1x_2\cdots$ and $\theta$ as its external angle.
		For any $0,1$ finite length word $w = w_1w_2\cdots w_n$, 
		$l^{\alpha, t}_w(\theta) = w_1\cdots w_n x_1x_2\cdots$.
	\end{lem}
	\begin{proof}
		This is straightforward by the definition of $l$.
		Recall that we attach a superscript $0$ or $1$ if the branch is defined on $A_\alpha$ or $B_\alpha$ (defined in \ref{defn:kneading_seq}), respectively.
	\end{proof}
	
	Now construct the lamination model for long chord $S_\lambda^\alpha := \frac{\alpha}{2}\frac{\alpha + 1}{2} = \dot{\alpha}\ddot{\alpha}$.
	Here $\lambda$ implies the empty word.
	Denote $S^{\alpha, t}_w := l^{\alpha, t}_w(S_\lambda^\alpha)$ for finite $0,1$ word $w$.
	As $\alpha$ is periodic, $l^{\alpha, t}_{tv^\alpha}(\dot{\alpha}) = \dot{\alpha}$ and $l^{\alpha, t}_{(1-t)v^\alpha}(\ddot{\alpha}) = \ddot{\alpha}$.
	
	To be precise, we define the $\alpha$ regular word $w\in \{0,1\}^*$.
	\begin{defn}[$\alpha$-regular word]
		Fix a periodic $\alpha \in S^1$.
		$w\in \{0,1\}^*$ is called $\alpha$-regular if $w$ does not end with $0v^\alpha$ nor $1v^\alpha$.
	\end{defn}

	We call \emph{step $k$} lamination generated by $\alpha$ if we collect all the $j$th preimage of the long chord $S^\alpha_\lambda$ for all $j < k$.
    Also we denote $\mathcal{L}_\alpha$ as a collection of all chord $S^{\alpha, t}_w$, $w\in\{0,1\}^*$ and $t\in \{0,1\}$.
	By construction, it became an invariant under $h$, both forward and backward.
    See the figure \ref{fig:deg2lam_ex}.
	
	Observe that $S^{\alpha,0}_{0v^\alpha}, S^{\alpha,1}_{0v^\alpha}, S^\alpha_\lambda$ form a triangle and similarly,  $S^{\alpha,0}_{1v^\alpha}, S^{\alpha,1}_{1v^\alpha}, S^\alpha_\lambda$ also form a triangle.
	We denote them $\Delta_{0v^\alpha}$ and $\Delta_{1v^\alpha}$, respectively.
	Under the inverse map of $h$, there are infinitely many triangle for each finite word $u \in \{0,1\}^*$, mutually disjoint interiors and possibly share vertices/edges.
	We call these as a triangle gap $\Delta_{utv^\alpha}, t\in\{0,1\}$, named after the label of chord.
	More precisely, three sides of $\Delta_{utv^\alpha}$ are $S^{\alpha, 0}_{utv^\alpha}, S^{\alpha, 1}_{utv^\alpha}, S^{\alpha, t}_{u}$.
	
	\begin{rmk}
		Let $u\in \{0,1\}^*$.
		Consider two triangle gap $\Delta_{utv^\alpha}$ and $\Delta_{utv^\alpha 0v^\alpha}$.
        They share a side $S^{\alpha, 0}_{utv^\alpha}$.
		Similarly, $\Delta_{utv^\alpha}$ and $\Delta_{utv^\alpha 1v^\alpha}$ share a side $S^{\alpha, 1}_{utv^\alpha}$.
		So, for each $\alpha$-regular word $w$, infinitely many triangles are attached from the $\Delta_{wtv^\alpha}$.
        Iteratively, it will form an infinite binary tree of triangles.
		\emph{i.e.,} if one has $0,1$ one sided sequence $x_1x_2\cdots$, then there is an associated path of triangle gaps $\Delta_{wx_1v^\alpha},\Delta_{wx_1v^\alpha x_2v^\alpha}, \cdots$ for each $\alpha$ regular word $w$.
        See the figure \ref{fig:triangle_gaps_remark}.
	\end{rmk}
	
	\begin{figure}
		\centering
		\begin{subfigure}{0.45\textwidth}
			\centering
			\includegraphics[width=.9\linewidth]{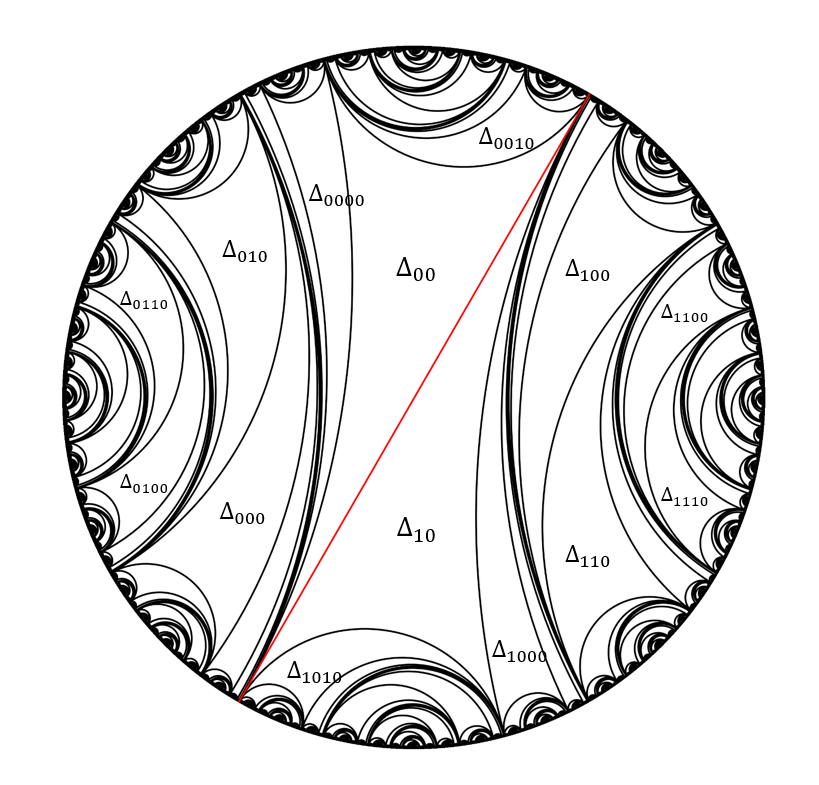}
		\end{subfigure}
		\begin{subfigure}{0.45\textwidth}
			\centering
			\includegraphics[width=.9\linewidth]{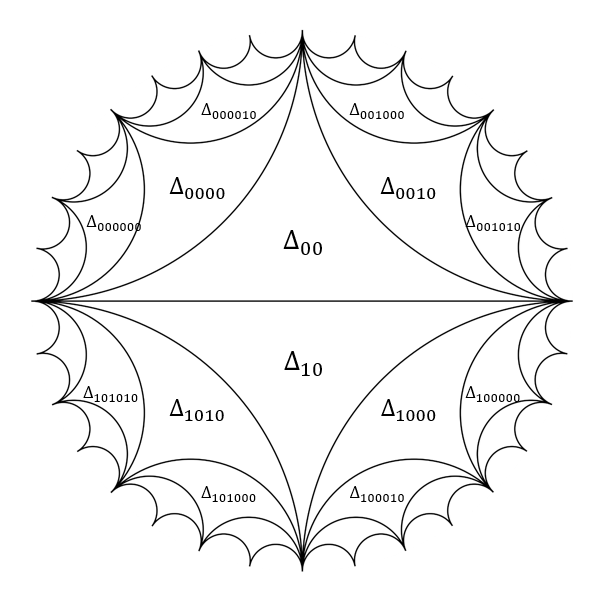}
		\end{subfigure}
		\caption{Here we have $\alpha = \frac{1}{3}$ and some indexes of triangle gaps. 
        If we focus on the triangle gaps starting from the long chord(red line in the left figure), then it forms a infinite binary tree.}
        \label{fig:triangle_gaps_remark}
	\end{figure}
	
	To distinguish sides of given triangle gap, we adopt the following notion.
	\begin{defn}[Characteristic symbol]\label{def:characteristic_symbol}
		Let $\alpha$ be a non-zero periodic in $S^1$.
		The symbol $t$ with $d(S^{\alpha, t}_{0v^\alpha}) < 1/4$ is called characteristic symbol of $\alpha$, denoted as $e^\alpha$.
	\end{defn}
	Note that $S^{\alpha, t}_{0v^\alpha}$ is a side of the triangle gap $\Delta_{0v^\alpha}$.
	As the definition says, the characteristic symbol states which side of triangle (except the long chord) is shorter than the other.
	More generally, consider we have an $\alpha$-regular word $w$.
	Then from $S^\alpha_w$ there is a triangle gap $\Delta_{w0v^\alpha}$, whose sides are $S^{\alpha, 0}_{w0v^\alpha}, S^{\alpha, 1}_{w0v^\alpha}$ and $S^\alpha_w$.
	The characteristic symbol $e$ states $d(S^{\alpha, e}_{w0v^\alpha}) < d(S^{\alpha, 1-e}_{w0v^\alpha})$.
	
	But for the other triangle gap  $\Delta_{w1v^\alpha}$ by symmetry $d(S^{\alpha, 1-e}_{w1v^\alpha}) < d(S^{\alpha, e}_{w1v^\alpha})$.
	Therefore the characteristic symbol is a tool to detect which one is shorter and it works like as a parity between super/subscripts.
		
	There are many ways to find which $t\in\{0,1\}$ is a characteristic symbol of $\alpha$, we refer one way as below.
	\begin{lem}\label{lem:char_symbol_determined}
		Let $\alpha$ be a non-zero periodic in $S^1$.
		$e^\alpha = 0$ if $\dot{\alpha}\dot{\overline{\alpha}}$ separates $\ddot{\alpha}\ddot{\overline{\alpha}}$ and $0\in S^1$.
		Symmetrically, $e^\alpha = 1$ if $\ddot{\alpha}\ddot{\overline{\alpha}}$ separates $\dot{\alpha}\dot{\overline{\alpha}}$ and $0\in S^1$.
	\end{lem}

	\begin{proof}
		This is a proposition 2.41 in \cite{keller2007invariant}.
	\end{proof}

	$S^{\alpha, e^\alpha}_{u0v^\alpha}$ for some $u \in \{0, 1\}^*$ became smaller and degenerate eventually as $|u|$ grows, since the subscript $e^\alpha$ always choose the shorter side.
	
	\subsection{Limit of laminations}
	\begin{defn}[Accumulation of chords]
		Suppose $x_iy_i$ is a sequence of chords with $x_i \to x$ and $y_i \to y$ in a usual topology of $S^1$.
		Then we say $x_iy_i$ accumulates to $xy$.
	\end{defn}
	Now can consider the limit lamination of given lamination $\mathcal{L}_\alpha$.
	\emph{i.e.,} the set of all accumulation chords of $\mathcal{L}_\alpha$.
	We denote it $\partial \mathcal{L}_\alpha$.
	It is easy to check that $\partial\mathcal{L}_\alpha$ is again an invariant lamination under $h$.
	
	Since $d(S^{\alpha, e^\alpha}_{u0v^\alpha})$ goes to $0$, taking the limit of $S_w^{\alpha, e}$ it degenerates eventually. 
	Thus we have the following corollary.
	\begin{cor}\label{cor:accumulate(1-e)}
		$\partial \mathcal{L}_\alpha$ for non-zero periodic $\alpha\in S^1$ is a limit of $S^{\alpha, 1-e}_w$ for all $w\in \{0,1\}$.
	\end{cor}
	With the previous remark, in the collection of $S^{\alpha, 1-e}_w$ every chord whose one endpoint is $\alpha$ has a form of $S^{\alpha, 1-e}_{(v^\alpha(1-e))^nv^\alpha}$ and they are associated by one binary path $(1-e)(1-e)\cdots$.
	Thus the other endpoint is $l^\alpha_{(v^\alpha(1-e))^nv^\alpha}(\ddot{\alpha})$.
	Furthermore since it follows $(1-e)$ marker only, the triangle path became thinner and thus $\alpha, l^\alpha_{v^\alpha}(\ddot{\alpha}), \cdots, l^\alpha_{(v^\alpha(1-e))^nv^\alpha}(\ddot{\alpha}), \cdots$ are positively ordered. 
	By corollary \ref{cor:accumulate(1-e)}, it accumulates and converges.
	We adopt the following definition.
	\begin{defn}[Associated periodic point]\label{def:associated_periodic_point}
		Let $\alpha$ be a non-zero periodic in $S^1$.
		The point $\overline{\alpha} := \lim_{n\to \infty} l^\alpha_{(v^\alpha(1-e))^nv^\alpha}(\ddot{\alpha})$ is well-defined and called \emph{associated periodic point with} $\alpha$. 
	\end{defn}
	In other words, $\overline{\alpha}$ is an endpoint of the accumulation chord whose the other endpoint is $\alpha$. 
	The associated periodic point $\overline{\alpha}$ shares properties of $\alpha$ as follows.
	\begin{prop}\label{prop:same_rep_word_v}
		$PER(\alpha)=PER(\overline{\alpha})$ and $v^\alpha = v^{\overline{\alpha}}$.
	\end{prop}	
	We will not prove the proposition.
	Instead of the proof we remark some properties of associated periodic points.
	\noindent
	\begin{rmk}
		\begin{enumerate}~
			
			\item For a non-zero periodic $\alpha$ in $S^1$, the associated periodic point $\overline{\alpha}$ of $\alpha$ and its companion angle coincide.
			Hence, Proposition \ref{prop:same_rep_word_v} holds and moreover $\overline{\overline{\alpha}} = \alpha$. (See corollary \ref{cor:assoc = companion}.)
			\item Since the angle doubling map $h$ preserves the cyclic order, we conclude that for $w\in\{0,1\}^*$,
			\[
			    \lim_{n\to \infty} l^\alpha_{w(v^\alpha(1-e))^nv^\alpha}(\ddot{\alpha}) = l^{\alpha, 1-e}_w(\overline{\alpha})
			\]
			and hence 
			\[
			  \lim_{n\to \infty} S^{\alpha, 1-e}_{w(v^\alpha(1-e))^nv^\alpha} = l^{\alpha, 1-e}_w(\alpha\overline{\alpha})
			\]
			This is just a preimage argument for the accumulation chord $\alpha\overline{\alpha}\in\partial\mathcal{L}_\alpha$.
			
			\item 
			Under this observation we give a nice formula for $\overline{\alpha}$.
			Let $m = \PER(\alpha)$.
			\[   
			    \overline{\alpha} = \alpha + \frac{2^m}{2^m - 1}\big(l^\alpha_{v^\alpha}(\ddot{\alpha}) - \alpha\big)
			\]
		\end{enumerate}
	\end{rmk}
	$\dot{\alpha}\dot{\overline{\alpha}}$ and $\ddot{\alpha}\ddot{\overline{\alpha}}$ are the preimage of $\alpha\overline{\alpha}$, and one of it is the chord whose endpoints are same with the arc $A_{n}$ in the orbit portrait, where $n = \PER(\alpha)$.
	Hence we have a following.
	\begin{cor}\label{cor:1-step_test_for_characteristic_symbol}
		Let $(\alpha, \overline{\alpha})$ is a characteristic arc of period $n$.
		Then $e^\alpha = \alpha\times (2^n - 1) \mod 2$.
		Therefore, every narrow characteristic arc has its characteristic symbol $e^\alpha = 1$.  
	\end{cor}
	\begin{proof}
		The proof is exactly same with lemma \ref{lem:narrow_starts_with_odd}.
		Suppose $\alpha = \frac{a}{2^n-1}$.
		If $a$ is even, $\ddot{\overline{\alpha}} = \overline{\alpha}/2 > \alpha/2 = \dot{\alpha}$.
		Thus $\ddot{\alpha}\ddot{\overline{\alpha}}$ and $0\in S^1$ and $e^\alpha = 0$.
		Symmetric proof for odd $a$.
	\end{proof}
	
	From the chord $\alpha\overline{\alpha}$ in the limit lamination $\partial \mathcal{L}_\alpha$, we can generate whole $\partial \mathcal{L}_\alpha$.

    \begin{figure}
        \centering
        \includegraphics[width = .5\textwidth]{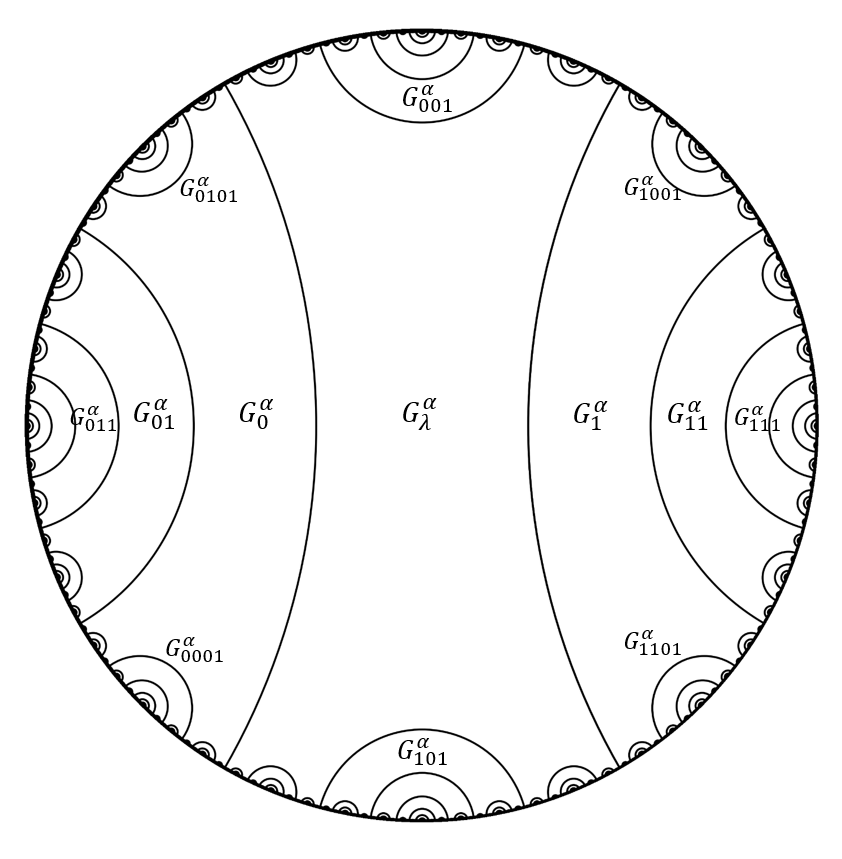}
        \caption{$\partial \mathcal{L}_\alpha$ for $\alpha = \frac{1}{3}$. Note that all the gaps are infinite gaps, whose intersection with $S^1$ is a Cantor set.}
        \label{fig:0.333_limitlam_index}
    \end{figure}
 
	\begin{thm}
		Let $\alpha$ be a non-zero periodic in $S^1$ and $e = e^\alpha$ be the characteristic symbol of $\alpha$.
		Then 
		\[
			\partial \mathcal{L}_\alpha = \text{ Closure of } \big\{ l^{\alpha, 1-e}_w (\alpha\overline{\alpha}) ~|~ w\in \{0,1\}^* \big\}
		\]
	\end{thm}
    \begin{proof}
        This is a proposition 2.42 in \cite{keller2007invariant}.
    \end{proof}
	
	One can easily guess that all gaps in $\partial \mathcal{L}_\alpha$ are made up of collection of triangles which start at a chord in $\mathcal{L}_\alpha$.
	As in the previous remark, triangles form an infinite binary trees.
	Together with the observation of correspondence between triangle path and binary sequence, we deduce the following.
	\begin{prop}\label{prop:infinite_gap_encoding}
		Suppose $\alpha$ and $e$ be as above and $m = PER(\alpha)$.
		Then each $\alpha$-regular word $w\in \{0,1\}^*$ there is a corresponding infinite gap $G_w^\alpha$ in $\partial \mathcal{L}_\alpha$, whose boundaries are $l^{\alpha, 1-e}_{ws_1v^\alpha s_2v^\alpha \cdots s_n}$ for $s_i \in \{0,1\}$.
	\end{prop} 
	We call $G^\alpha_{v^\alpha}$ as a critical value gap, since it contains $\alpha$ as its boundary gap, which is an external angle of critical value.
	Also, $G^\alpha_\lambda$, a infinite gap corresponding to empty word $\lambda$ is called critical point gap.
	\begin{proof}
		Fix an $\alpha$-regular word $w$.
		Then the gap $G^\alpha_w$ is associated to the infinite binary tree, which is a collection of triangle gaps starting from $S^\alpha_w$.
		\emph{i.e.,} $G^\alpha_w$ can be considered as a union of $\Delta^\alpha_{ws_1v^\alpha s_2v^\alpha\cdots s_nv^\alpha}$, with $s_1\cdots s_n \in\{0,1\}^*$. 
		Furthermore, each chord accumulated by the long sides of such triangles became a boundary of given gap $G^\alpha_w$.
		Hence we are done.
	\end{proof}
	
	Choose one infinity gap $G^\alpha_w$ of $\alpha$-regular word $w$.
	By proposition \ref{prop:infinite_gap_encoding}, this gap is bounded by infinitely many chords $l^{\alpha, 1-e}_{ws_1v^\alpha s_2v^\alpha \cdots s_n}$.
	So the infinity gap can be restated as a Cantor set, starting from the circle and delete all the points out which lie behind the chords $l^{\alpha, 1-e}_{ws_1v^\alpha s_2v^\alpha \cdots s_n}$.
	Hence any surviving point $p\in G^\alpha_w$ has an itinerary sequence of form $ws_1v^\alpha s_2v^\alpha\cdots$ where $s_1s_2\cdots$ is a $0$-$1$ sequence.
	Notice that if given itinerary sequence ends with $\overline{0v^\alpha}$ or $\overline{1v^\alpha}$, then it converges to an endpoint of some boundary chord.
	According to the observation above we get the following.
	
	\begin{prop}\label{prop:endpt_of_inf_gap}
		$p\in G^\alpha_w$ if and only if $I^\alpha(p)$ ends with $0v^\alpha$-$1v^\alpha$ sequences.
		Furthermore, if $I^\alpha(p)$ ends with $\overline{0v^\alpha}$ or $\overline{1v^\alpha}$, then $p$ is an endpoint of the boundary chord.
	\end{prop}
	The converse of the latter statement does not hold at all because the endpoint of the boundary chord might be the preimage of $\alpha$, whose itinerary sequence contains $*$ eventually.
		
	\subsection{Translation to quadratic polynomials}
	In this section $\alpha \in S^1$ is nonzero periodic under $h$.
	For such $\alpha$ there is an associated angle $\overline{\alpha}$.
	As we aforementioned in the preliminary, such point is called a \emph{companion angle}.
	Thus two external rays of angle $\alpha, \overline{\alpha}$ in the parameter space land at the same point $r_{M_\alpha}$, which is a root of a hyperbolic component $M_\alpha$.
	From now on, we drop $M$ and denote the center and root as $c_\alpha$ and $r_\alpha$ if $\alpha$ is well understood.
	\emph{i.e.,} $z\mapsto z^2 + r_\alpha$ has a parabolic fixed point for any periodic $\alpha$.
	
	\begin{cor}[Corollary 4.15 in \cite{keller2007invariant}]\label{cor:assoc = companion}
		For periodic $\alpha\in S^1$ the parameter ray of angle $\alpha, \overline{\alpha}$ land at the same root point $r_M$ for some $M$.
		Therefore, the associated angle $\overline{\alpha}$ of nonzero periodic $\alpha$ is identical to the companion angle.
	\end{cor}
	
	At $r_M$, the Julia set of $z^2 + r_M$ is connected.
	Let $\alpha, \overline{\alpha}$ be a pair corresponds to $M$.
	Suppose $\approx_\alpha$ be an equivalence relation on $S^1$ satisfying that 
	\[
		\theta_1 \approx_\alpha \theta_2 \quad \Leftrightarrow \quad \text{ Dynamic rays of angle }\theta_1, \theta_2 \text{ lands at the same point.}
	\]
	If we draw a line/polygon whose endpoints/vertices are in the same class of $\approx_\alpha$, then it became a lamination.
	We call this a \emph{landing pattern} of $\alpha$, or the \emph{pinched disc model of $K_\alpha$ ($= K_{\overline{\alpha}}$)}.
	We remark the following theorem.
	\begin{thm}[Corollary 4.2 in \cite{keller2007invariant}]\label{thm:landing_pattern=limit_lam}
		Suppose $\alpha$ is nonzero periodic in $S^1$.
		Then the landing pattern of $\alpha$ is equal to $\partial\mathcal{L}_\alpha$.
	\end{thm}
	Suppose a paramter $c$ follows along the parameter ray $\mathcal{R}_\alpha$ and finally lands on the root $r_M$.
	At each moment $c$ lies on the $\mathcal{R}_\alpha$ the lamination model of $J_c$ is the one generated by long chord $\frac{\alpha}{2}\frac{\alpha+1}{2} = \dot{\alpha}\ddot{\alpha}$.
	If we pinch it from the long chord and its 1st preimages and 2nd and so on, then the limiting process would be a Julia set of $z\mapsto z^2 +c_\alpha$, where $c_\alpha$ is a parameter outside of Mandelbrot set whose parameter angle is $\alpha$.

	When $c$ lands at $r_M$ eventually, we choose a lamination model of $J_{r_M}$ to be its landing pattern.
	Then such choice of limit lamination enable us to consider the landing of parameter as in the laminational model as below.
	\emph{i.e.,}
	\[
		\lim_{c \in \mathcal{R}_\alpha,~ c \to r_M} \mathcal{L}_\alpha = \partial\mathcal{L}_\alpha
	\]
	The limit on the lefthand side is by the topology derived from the accumulation of chords.
	It gives one way to understand the boundary of shift locus of degree $2$ in the sense of laminations.
		
	\section{Action on itinerary sequences} \label{sec:5-action_on_iti_seq}		
		
	We now introduce the quadrilateral lamination $\mathcal{Q}$.
	
	\begin{defn}[quadrilateral lamination] \label{defn:quadrilateral_lam}
		Suppose $\overline{\alpha}$ be an associated angle of $\alpha$.
		The center gap associated to $(\alpha, \overline{\alpha})$, denoted as $Q^{(\alpha,\overline{\alpha})}_\lambda$, is a gap whose endpoint is $\alpha/2, (\alpha+1)/2, \overline{\alpha}/2, (\overline{\alpha}+1)/2$.
		\emph{i.e.,} long chords $S^\alpha_\lambda, S^{\overline{\alpha}}_\lambda$ of $\mathcal{L}_\alpha$ and $\mathcal{L}_{\overline{\alpha}}$ are diagonals of the center gap $Q^{(\alpha,\overline{\alpha})}_\lambda$.
		
		\emph{Quadrilateral lamination} associated to $(\alpha, \overline{\alpha})$, denoted as $\mathcal{Q}^{\alpha}$ is the collection of preimages of $Q = Q^{(\alpha,\overline{\alpha})}_\lambda$ under $h$.
	\end{defn}
	
	The center gap is actually a rectangle, since its diagonals are diameters.
	Hence the there are two types of sides, one is longer than the other. 
    We provide a figure \ref{fig:quad_lam} of quadraliteral lamination $\mathcal{Q}^\alpha$ for $\alpha = 1/3$.
	
	Remark that $\alpha$ and its associated angle $\overline{\alpha}$ form a characteristic arc. 
	Also, $h(Q)$ is equal to characteristic arc $\alpha\overline{\alpha}$ and the shorter sides are the preimage $l^{\alpha, e^\alpha}$ and the other is the preimage of $l^{\alpha, 1-e^\alpha}$ because of their length.
	Hence we have the following.
	\begin{prop}\label{prop:lim_lam_is_a_subset}
		All chords in $\partial\mathcal{L}_\alpha$ are in the $\mathcal{Q}^{\alpha}$.
		The others has a form of $l^{\alpha, e}_w(\alpha\overline{\alpha})$.
		\emph{i.e.,} every chord in $\mathcal{Q}^\alpha$ has a form of $l^{\alpha, t}_w(\alpha\overline{\alpha})$, $t\in\{0,1\}$ and $w\in \{0,1\}^*$.
	\end{prop}
	\begin{proof}
		Note that the quadrilateral gap $Q^\alpha$ is surrounded by $4$ distinct chords, one pair of long sides and the other pair of short sides.
		Since $\alpha/2, (\alpha+1)/2$ are antipodal to each other, the long side has to be a form of $\frac{\alpha}{2}\frac{\overline{\alpha}}{2}$ or $\frac{\alpha}{2}\frac{\overline{\alpha}+1}{2}$ and the other one must be the short side.
		Without loss of generality, suppose $\frac{\alpha}{2}\frac{\overline{\alpha}}{2}$ is a long side.
		Since it is a preimage of $\alpha\overline{\alpha}$ and it is longer than the other preimage $\frac{\alpha}{2}\frac{\overline{\alpha}+1}{2}$, it coincides to $l^{\alpha, 1-e}_s(\alpha\overline{\alpha})$.
		Exactly same argument will give $l^{\alpha, e}_s(\alpha\overline{\alpha}) = \frac{\alpha}{2}\frac{\overline{\alpha}+1}{2}$.
		Here $s$ is just a constant only to specify where the chord lies beyond $S^\alpha$ or not.
		Hence all chords are the preimages of $l^{\alpha, t}_w(\alpha\overline{\alpha})$ with $t\in\{0,1\}$ and $w\in \{0,1\}^*$.	
	\end{proof}
	We call every chord in $l^{\alpha, e^\alpha}_w(\alpha\overline{\alpha})$ as \emph{short chord} and except those we denote them as \emph{long chord}.
	Therefore, $\mathcal{Q}^\alpha$ is a lamination generated by $\alpha\overline{\alpha}$.
	
	\begin{cor}
		$\mathcal{Q}^\alpha$ is a lamination.
	\end{cor}

	Adding short chords gives another description of infinite gaps in $\partial \mathcal{L}_\alpha$.
	Choose the infinite gap $G_\lambda^\alpha$ at the center.
	It is bounded by the long chords.
	Adding the pair of short chords in the center gap divides the infinite gap $G^\alpha_\lambda$ into $3$ pieces, center gap and $2$ more pieces.
	If we add more short chords, $2$ more quadrilateral gaps are attached on both short chords of $Q^\alpha$.
	Iteratively, $G_\lambda^\alpha$ is comprised of infinitely many quadrilateral gaps, which forms an infinite binary tree.
    See the figure \ref{fig:quad_lam} also.
	
	We assign each quadrilateral gap $Q_w$ with a word $w \in \{0,1\}^*$.
	Such word is a prefix, \emph{i.e.,} $l^\alpha_{w}(Q) = Q_w$.
	By the observation, if $w = w'0v^\alpha$ or $w'1v^\alpha$, then the quadrilateral gap $Q_w$ is located right next to the $Q_{w'}$, sharing one of short chords of $Q_{w'}$.
	
	\begin{figure}
	    \centering
	    \includegraphics[width=.5\textwidth]{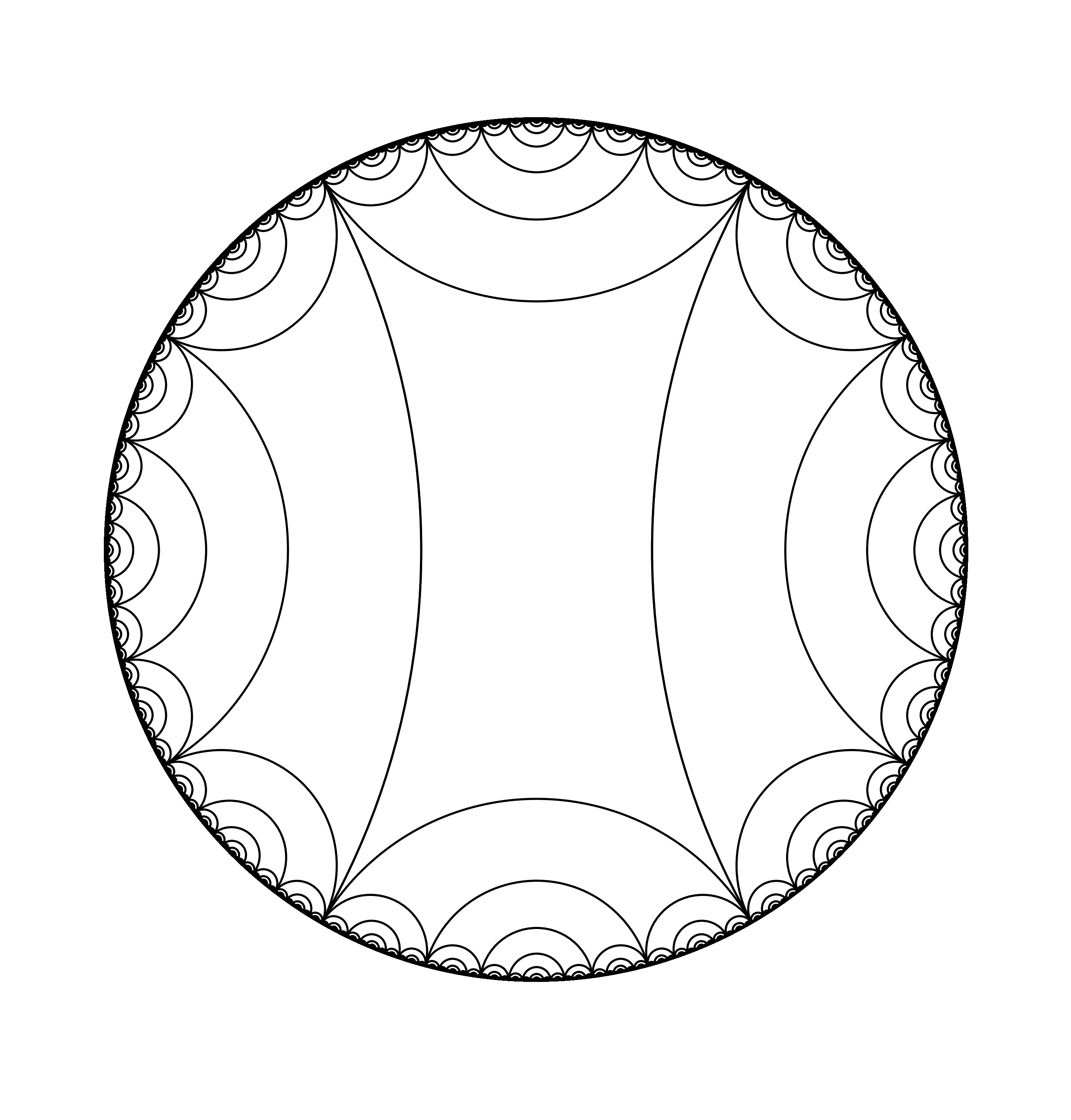}
	    \caption{$\mathcal{Q}^\alpha$ with $\alpha = \frac{1}{3}$. Compare with figure \ref{fig:0.333_limitlam_index} also.}
	    \label{fig:quad_lam}
	\end{figure}
	
	As in the definition, all quadrilateral gaps in $\mathcal{Q}$ have diagonals which are the preimages of $S^\alpha_\lambda$ and $S^{\overline{\alpha}}_\lambda$.
	Thus if $\alpha$ changes to $\overline{\alpha}$, then the only thing changes is the diagonal.
	This implies if a point $p$ lies behind the short chord of $Q^\alpha$, the first symbol $x$ of $I^\alpha(p)$ changes to $1-x$ when it comes to $I^{\overline{\alpha}}(p)$.
	Hence we get the following,
	\begin{lem}\label{lem:symbol_change}
		Suppose $p$ is not an endpoint of $\partial \mathcal{L}_\alpha$.
		Then $n$th symbol $s$ of $I^\alpha(p)$ changes to $(1-s)$ if and only if $p$ lies behind the $(n-1)$th preimage of the short chord. 
	\end{lem}
	\begin{proof}
		After taking $h^{(n-1)}$, $n$th symbol of $I^\alpha(h^{(n-1)}(p))$ became the first one, since $h$ acts on the itinerary sequence as a one-sided shift.
		As diagonals of $Q$ are the long chords, the point lies on the short side if and only if the first symbol changes.
		Thus we are done.
	\end{proof}
	
	By proposition \ref{prop:lim_lam_is_a_subset}, all the infinity gap in $\partial\mathcal{L}_\alpha$ can be decomposed into quadrilateral gap $Q^\alpha$ and its preimages, by $l^{\alpha, e}_w(\alpha\overline{\alpha})$.
	We call every chord of form $l^{\alpha, 1-e}_w(\alpha\overline{\alpha})$ as a \emph{long chord} and $l^{\alpha, e}_w(\alpha\overline{\alpha})$ as a \emph{short chord}.
	
	If we focus on diagonals of pre-quadrilateral gaps, $\alpha$ changes to $\overline{\alpha}$ implies all the diagonal are exchanged and it means the symbol $0, 1$ in front of $v^\alpha = v^{\overline{\alpha}}$.
	
	\begin{thm}\label{thm:0v1v-seq_rule}
		Suppose $p$ is a point in $G^\alpha_w$ for some $\alpha$-regular word $w$ and $I^\alpha(p) = ws_1v^\alpha s_2v^\alpha \cdots$ for some $s_i\in\{0,1\}.$
		Then $I^{\overline{\alpha}}(p) = w'(1-s_1)v^\alpha(1-s_2)v^\alpha \cdots$.
		\emph{i.e.,} $0v^\alpha, 1v^\alpha$ swap each other in the $0v^\alpha$-$1v^\alpha$ sequence part.
		Here $w'$ is another $\alpha$-regular word.
	\end{thm}
	\begin{proof}
		The first statement is obvious by lemma \ref{lem:symbol_change}.
		Also the actual position of $p$ does not move, it still lies on the same infinity gap, but different encodings with respect to $\partial\mathcal{L}_{\overline{\alpha}}$.
	\end{proof}
	Despite the encoding of infinity gap differs as $\alpha$ changes to $\overline{\alpha}$, its depth does not change and therefore $|w| = |w'|$.
	\emph{i.e.,} $\varphi_\alpha$ is a level preserving map.
	
	So the only thing we left is how the encoding of chosen infinite gap varies whenever $\alpha$ alters to $\overline{\alpha}$.
	In other words, we need a rule for $\alpha$-regular words.
	Recall that the only situation that the symbol changes is when the point $p$ or gap $G$ lies on some preimage of short side of $Q^\alpha$.
	Therefore, whenever the infinite gap in $\partial\mathcal{L}_\alpha$ lies behind the short side, some symbols in the encoded $\alpha$-regular word $w$ changes.
	
	For the narrow $\alpha$, we propose a rule.
	Recall that the characteristic arc $(\alpha, \overline{\alpha})$ is narrow if $\overline{\alpha} - \alpha=\frac{1}{2^{\PER(\alpha)}-1}$. 
	By proposition \ref{prop:narrow = no nesting}, narrow implies there are no nesting.
	Moreover, the narrow arc does not nest any other characteristic arcs of period strictly less than the period of itself.
	Recall also that every narrow arc $\alpha\overline{\alpha}$ admits its characteristic symbol $e^\alpha = 1$. 
	
	We first start with the lemma.
	\begin{lem}\label{lem:narrow_follows_v_alpha}
		Suppose $\alpha$ is narrow.
		Let $p$ lie behind one short side of the center gap of $\alpha$.
		Then $I^\alpha(p)$ starts with $0v^\alpha$ or $1v^\alpha$ and the first symbol depends on which short side $p$ lies behind.
	\end{lem}
	\begin{proof}
		Since $\alpha$ is narrow, every point lies behind the chord $\alpha\overline{\alpha}$ follows $v^\alpha$.
		This is because under angle doubling such points keep lie behind $A_1, \cdots, A_{n-1}$.
		Two short sides of the center gap is the preimage of the characteristic arc.
		Thus the first symbol of itinerary sequence is $0$ or $1$ followed by $v^\alpha$, and the first symbol only depends on which preimage $p$ belongs to.
	\end{proof}
	
	\begin{prop}\label{prop:narrow_a-regular_prefix}
		Suppose $I = \alpha\overline{\alpha}$ is narrow.
		Then the infinite gap of the shortest prefix among all gaps which lie behind the short edge $l^{\alpha, e^\alpha}_{tv^\alpha}$ is $G^{\alpha}_{tv^\alpha e^\alpha}$, $t\in\{0,1\}$.
	\end{prop}
	\begin{proof}
		We only prove for the case $t = 0$.
		The other case holds by symmetry.
		Let $Q$ be the center gap.
		By the observation, the quadrilateral gap $Q_{0v^\alpha}$ with prefix $0v^\alpha$ shares the short chord of the center gap.
		Let $l_1$ be a chord of $Q_{0v^\alpha}$ opposite to the short chord, and choose any quadrilateral gap $Q$ which lies behind $l_1$.
		Then since it lies behind the short side of the center gap, by lemma \ref{lem:narrow_follows_v_alpha} the prefix of $Q$ starts with $0v^\alpha$.
		Take $h^{n}$ to $Q$, then the endpoints of $h^n(Q)$ is contained in the arc $A_n$.
		Moreover, since $h^{n}(l_1) = \ddot{\alpha}\ddot{\overline{\alpha}}$, all endpoints lie behind the chord $\ddot{\alpha}\ddot{\overline{\alpha}}$.
		Thus any quadrilateral gap lie behind the short chord starts with $tv^\alpha e^\alpha$.
	\end{proof}
	Remark that the infinite gap $G^\alpha_{tv^\alpha(1-e^\alpha)}$, $t\in \{0,1\}$ do not lie behind the short chord of the center gap.
    See the figure \ref{fig:0.333_limitlam_index}.
		
	By using this, we can now detect how the prefix of infinite gap changes.
	Let $w$ be $\alpha$-regular word and $n = \PER(\alpha)$.
	We begin with checking the first $n$ symbols form a $t_1v^\alpha$.
	If not, we just slide one symbol over.
	If so, then we check the next $n$ symbols whether they form $t_2v^\alpha$ or not.
	We do until they do not form $tv^\alpha$.
	Suppose it ends at $t_kv^\alpha$.
	The next step is to check the next following symbol.
	If the symbol is $e^\alpha$, then all such $t_iv^\alpha$ changes to $(1-t_i)v^\alpha$.
	This is because it indicates such gap lies behind the boundary chord $l^{\alpha, 1-e^\alpha}_{t_1v^\alpha\cdots t_kv^\alpha}(\alpha\overline{\alpha})$ of the center gap.
	
	What if the symbol is not $e^\alpha$?
	Then we go back to the previous $v^\alpha$, and check what symbol follows after that.
	In this case we have the sequence $\cdots t_{k-1}v^\alpha t_kv^\alpha (1-e^\alpha) \cdots$, so we check whether $t_k$ is $e^\alpha$ or not.
	If not, then we go back again and again.		
	
	For example, let $\alpha = \frac{1}{7}$. 
	Then $v^\alpha = 00$ and $e^\alpha = 1$.
	If we have a $\alpha$-regular prefix $01000000001010010$, it changes to another prefix as below.
	\[
		0|1|0|0|000~0001|0|1001|0 \Rightarrow 0|1|0|0|100~1001|0|0001|0
	\]
	Note that the prefix has $100$ from the second symbol, but $0$ follows behind which is $1-e^\alpha$.
	Hence the second symbol does not change, and we slide one symbol over.
	
	Summarizing up, we have the following algorithm.	
	\begin{thm}\label{thm:rule_unified}
		Let $s = s_1s_2s_3\cdots$ be a $0$-$1$ sequence.
		Also, let $v = v^\alpha$ be the repetition word of $\alpha$, $e = e^\alpha$ is a characteristic symbol of $\alpha$ and $\PER(\alpha) = |v|+1$ be the period of $\alpha$.
		Then we define $\varphi_\alpha(s)$ as below.
		\begin{enumerate}
			\item $i = 0$.
			\item Set $w = s_i\cdots s_{i+|v|}$.
			\begin{enumerate}
				\item If $w = 0v$ or $1v$, set $i \leftarrow i+|v|+1$ and go back to (2).
				\item If not, but if there is a previous $0v$ or $1v$, check $s_i = e$. 
				\begin{enumerate}
					\item If $s_{i+|v|+1} = e$, swap all $0v$ and $1v$. 
					\item If $s_{i+|v|+1} \neq e$, trace back to the previous $0v$ (or $1v$) and go back to (b). 
				\end{enumerate}
			\end{enumerate}			
			\item If $w$ is neither $0v$ nor $1v$, then set $i \leftarrow i+1$ and go back to (2).
		\end{enumerate}

        Suppose $(\alpha, \overline{\alpha})$ is narrow.
        Choose $p \in J_\alpha$ and let $\theta$ be the dynamic angle of $p$.
        If $\theta$ is not precritical angle, then $\varphi_\alpha(I^\alpha(\theta)) = I^{\overline{\alpha}}(\theta)$.
	\end{thm}
	\begin{proof}
		We have seen for any point $p \in G^\alpha_w$ for any $\alpha$-regular word $w$.
		The remaining part is the case $p$ is the limit of infinite gaps.
		Then $I^\alpha$ does not end with $0v^\alpha$-$1v^\alpha$ sequences.
		Hence, by truncating from the beginning of $I^\alpha(p)$, it is approximated by the prefixes, which nests $p$.
		Therefore, $\varphi_\alpha(I^\alpha(p)) = I^{\overline{\alpha}}(p)$.
	\end{proof}
	
	We remark that if $\alpha$ is not narrow, then the algorithm fails.
	Because of the nonnarrowness, there is $m < PER(\alpha)$ which $A_m$ nests $A_1$.
	This implies that there exists an infinite gap $G^\alpha_w$ with shorter length $|w| < |tv^\alpha e^\alpha|$, and such gap may change their symbols.
	Also, if the point lie behind the nonnarrow arc, then the itinerary sequence need not begin with $v^\alpha$, but only the first $(m-1)$ symbols of $v^\alpha$.
 
	Another remark is that $v^\alpha$ has an ambiguity for some sequences, especially the itinerary sequence which ends with $\overline{0v^\alpha}$ or $\overline{1v^\alpha}$.
	For example, let $\alpha = \frac{1}{7}$. 
	Then $v^\alpha = 00$ and $e^\alpha = 1$.
	Now consider the sequence $s = \overline{0} = 0000\cdots$.
	By the rule, step (2) will not end and we have to distinguish at which $0v$-$1v$ sequence starts. 
	But the problem is that we cannot distinguish what is the prefix of this sequence because all 3 cases (the empty word, $0$, $00$) could be its prefix.
	Each case is obtained by approximation of prefixes.
	For empty word prefix case, it is approximated by $0$, $0001$, $000~0001$, $000~000~0001$, so that after taking $\varphi_\alpha$ it gives $\overline{100}$.
	
	In the next section we will resolve this ambiguity by constructing an equivalence relation.
	
	\section{Equivalence relations on $\Sigma_2$} \label{sec:6-equiv_rel}
	Definition of itinerary sequence avoids some sequence in $\Sigma_2$ is not realized.
	For example, the $\overline{0} = 000\cdots$ cannot be realized by $I^\alpha(\theta)$, $\theta\in S^1$.
	This is because if we keep taking the inverse image so that each image contains an angle whose itinerary sequence starts with $0\cdots 0$, it converges to the set $\{1/3, 2/3\}$ and both angle contain $*$ in their sequence.
	 
	Hence the algorithm we propose is not defined on the whole $\Sigma_2$ and by the ambiguity $\varphi_\alpha$ is a multi-valued function.
	Also, because of detecting a piece of word and swap, it is far from shift invariant.
	
	Therefore, in this section our objectives are the followings.
	
	\begin{enumerate}
		\item Extend the angles so that $I^\alpha$ maps surjectively on $\Sigma_2$.
		\item Extend the domain of $\varphi_\alpha$ to $\Sigma_2$.
		\item Construct the equivalence relation on $\Sigma_2$ to make our map $\varphi_\alpha$ became a shift map.
	\end{enumerate} 
	
	\subsection{Extended angles}
	Consider $z^2 + 2$.
	The parameter angle of $c = 2$ is $0$, and since it is in $\mathcal{S}_2$ the Julia set is a Cantor set.
	Let $p_1, p_2$ be points in the Julia set corresponding to each sequence $\overline{0} = 000\cdots$, $\overline{1} = 111\cdots$, respectively.
	What are their external angles?
	Since those two points are fixed point under shift, the angle must be the fixed point of angle doubling map.
	Hence, both admits $0$ as their external angles.
	
	The main reason is when we extend B\"{o}ttcher coordinate into the whole $\mathbb{C} - J_P$ for shift polynomial $P$, we keep solving the equation $z\mapsto \sqrt{\phi(P^{\circ 2}(z))}$.
	Hence the branched region became smaller and smaller and it might converge to some points whose angles are pre-critical.
	As in the above example, the region that the first symbol of itinerary is equal to $0$ is an open interval $(0, 1/2)$, and the region that both first and second symbols are $0$ is an open interval $(0, 1/4)$.
	Keep doinig that we deduce that the region whose itinerary sequence starts with $\underbrace{0\cdots 0}_{n}$ is $(0, 1/2^n)$.
	\emph{i.e.,} The point $p$ has a dynamic angle of $0$, and such sequence of interval converges to $0$.
	But we can do the exact same for $1$, by taking $(1/2,1), (3/4, 1), (7/8, 1), \cdots$ and it converges to $1 = 0$.
	
	To distinguish those two points, we define the extended angle.
	\begin{defn}\label{defn:extended_angle}
		Let $P := P_\alpha(z) = z^2 + c$ is a shift polynomial satisfying that the parameter angle of $c$ is $\alpha$, which is periodic.
		Suppose $x \in J_P$ corresponds to $s = s_1s_2\cdots$ and let $\theta$ be the external angle of $x$.
		Then we endow a marker $+$ or $-$ on $\theta$ if it satisfies below.
		\begin{enumerate}
			\item $\theta^+$ if $\underset{\theta' \to \theta, \theta' > \theta}\lim I^\alpha(\theta') = s$,
			\item $\theta^-$ if $\underset{\theta' \to \theta, \theta' < \theta}\lim I^\alpha(\theta') = s$,
			\item $\theta$ if both limit (from the left and right) coincides.
		\end{enumerate}
		We call such angle $\theta^\pm$ as an \emph{extended angle}.
		Remark that if $\theta$ is not a preperiodic angle of $\alpha$, then the external angle does not have a marker.
	\end{defn}
	For example of $\alpha = 0$ as in the above, $0^+$ is an external angle of sequence $\overline{0}$, on the other hand $0^-$ for $\overline{1}$.
	
	Extended angle has an order inherited by the usual order in $S^1$, by considering $\theta^- < \theta < \theta^+$.
	With the ordering, $I^\alpha(\theta^\pm)$ is well-defined and compatible with the definition.
	Hence, we extend the domain as follows.
	\[
		\{ \theta \text{ is not preimage of $\alpha$} \} \cup \{ \theta^\pm \text{ is a preimage of }\alpha \text{, with a corresponding marker} \}.
	\]
	We denote the domain $\mathbf{E}_\alpha$.
	\begin{prop}\label{prop:extended_domain_of_I^alpha}
		Let $\alpha$ be nonzero periodic in $S^1$.
		Then $I^\alpha : \mathbf{E}_\alpha \to \Sigma_2$ is surjective.
	\end{prop}
	\begin{proof}
		It is straightforward from the definition.
		Choose any sequence in $\Sigma_2$.
		Since we start with the Julia set $J_\alpha$ of angle $\alpha$, there is a point $x\in J_\alpha$ which corresponds to $\alpha$.
		Then its extended angle is in $\mathbf{E}_\alpha$.
	\end{proof}
    With the definition, $\varphi_\alpha$ is a map which commutes the diagram.
	\[
		\begin{tikzcd}
			\mathbf{E}_\alpha \arrow[r, "Id"] \arrow[d, "I^\alpha"] & \mathbf{E}_\alpha \arrow[d, "I^{\overline{\alpha}}"] \\
			{\Sigma_2} \arrow[r, "\varphi_\alpha"]        & {\Sigma_2}                                
		\end{tikzcd}
	\]
	Still $\varphi_\alpha$ is far from injectivity.
	For example, consider $\alpha = 2/7$ and let $\theta_1 = \frac{1}{7}^+, \theta_2 = \frac{2}{7}^+, \theta_3 = \frac{4}{7}^+$.
	Then all three $I^\alpha(\theta_i) = \overline{0}$.
	However it is enable to distinguish the prefix if we know not only the sequence but the extended angle, by which infinite gap it belongs to.
	In the example $\theta_1 \in G^\alpha_\lambda, \theta_2 \in G^\alpha_{0}$, and $\theta_3 \in G^\alpha_{00}$.
    Their prefixes appear when each marker alters, such as $I^{\frac{2}{7}}\left( \frac{4}{7}^- \right) = \overline{001} = 00~\overline{100}$.
    This is described in the figure \ref{fig:rabbit_and_rays}.

    \begin{figure}
        \centering
        \includegraphics[width=.9\textwidth]{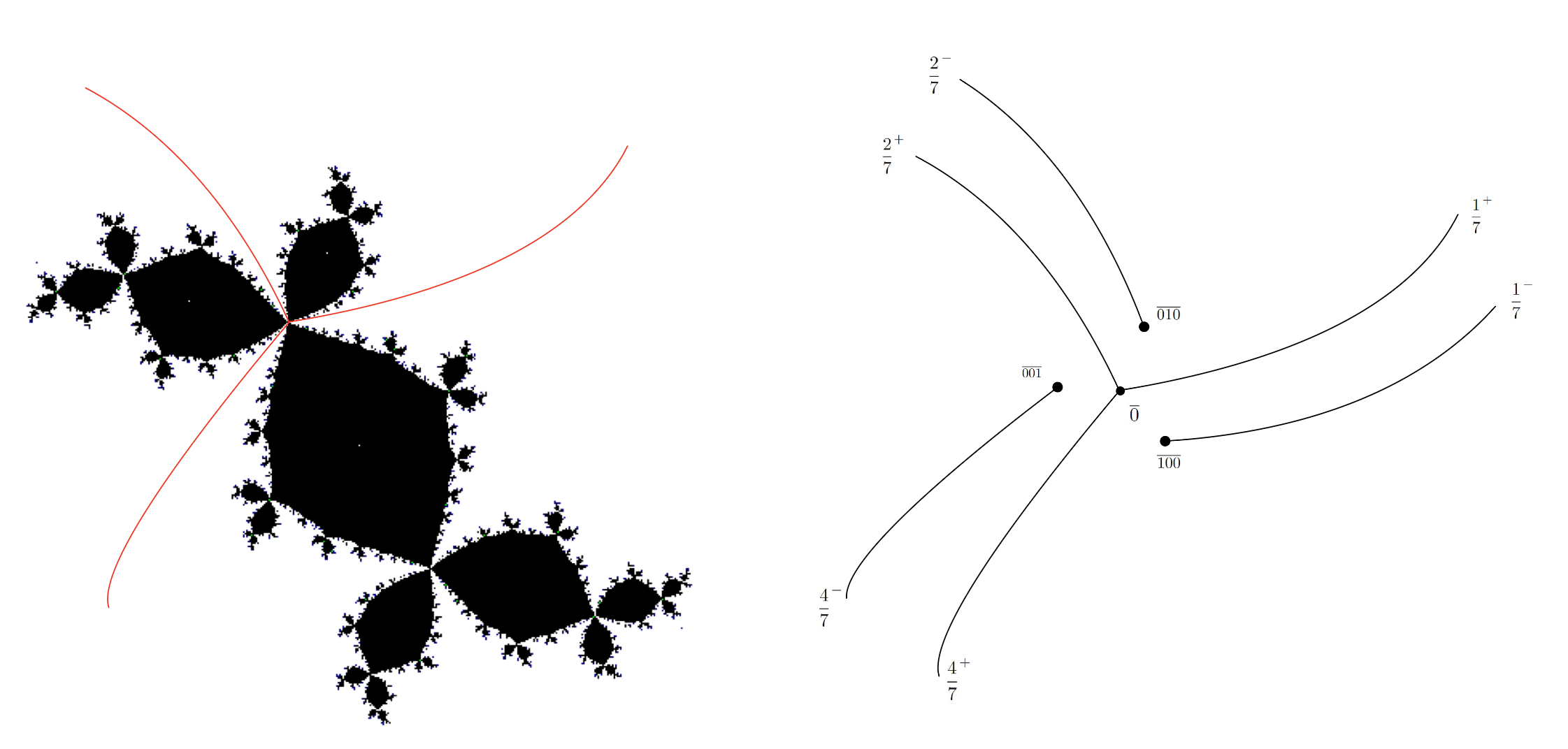}
        \caption{Figure on the left illustrates the dynamical plane and actual rays which lands at the fixed point of $z^2 + c$, where $c \approx -0.1226 +0.7449i$, a rabbit.
        Figure on the right is a schematic version of the left, before $c$ lands at the root.
        Black dots are points in Julia set.
        Suppose $\alpha = \frac{2}{7}$.
        As $c$ moves along $\mathcal{R}^\alpha$ and eventually lands at the root of the hyperbolic component, those $4$ points collapse into $1$ point.
        Here the dynamic rays should in fact bifurcate, but we draw it separately to distinguish easily. 
        Such bifurcation rays were discussed in \cite{atela1992bifurcations} also.}
        \label{fig:rabbit_and_rays}
    \end{figure}

	\subsection{Parabolic implosion}
	Yet $\varphi_\alpha$ is a multi-valued in the sense of mapping between sequences.
	In this section we will resolve this problem.
	
	We first define a Julia equivalence on $\mathbf{E}_\alpha$.
	This is again introduced by Keller in \cite{keller2007invariant}.
	\begin{defn}[Julia equivalence]\label{defn:Julia_equivalence}
		For periodic $\alpha\in S^1$ and $\theta, \theta'\in S^1$, we denote $\theta \approx^\alpha \theta'$ if there exists a finite sequence of points $\theta= z_1, z_2 \cdots, z_{n-1}, z_n = \theta'$ such that $z_iz_{i+1}$ are chords of $\partial \mathcal{L}^\alpha$.
	\end{defn}	
	One can easily verify that it is actually an equivalence relation.
	It can be translated into the language of itinerary sequence.
	
	\begin{prop}\label{prop:Julia_equiv_and_iti_seq}
		$x \approx^\alpha y$ $\iff$ $I^\alpha(x) = I^\alpha(y)$. 
	\end{prop}

	The equivalence is extended on the $\mathbf{E}_\alpha$ from $S^1$, by dropping the markings.
	By definition, $\theta_1 \approx^\alpha \theta_2$ if two external(dynamics) rays $\mathcal{R}_\theta, \mathcal{R}_{\theta'}$ lands at the same point of the Julia set of $z^2 + r_\mathcal{P}$, where $r_\mathcal{P}$ is the root of the hyperbolic component $\mathcal{P}$ corresponding to $\alpha$ and its companion angle.	
	Therefore the extension is exactly an equivalence relation of $\Sigma_2$.
	\begin{defn}\label{defn:equivalence_relation_on_Sigma_2}
		Let $s_1, s_2 \in \Sigma_2$ and $x_1, x_2 \in J_\alpha$ be the corresponding point.
		Suppose $\theta_1, \theta_2$ are extended angles of $x_1, x_2$, respectively.
		We define the equivalence relation $\sim$ as follows,
		\[
			s_1 \sim s_2 \quad \Leftrightarrow \quad \theta_1 \approx^\alpha \theta_2
		\]
	\end{defn}
	As $c$ follows the parameter ray $\mathcal{R}_\alpha$ and eventually lands at the root of hyperbolic components, some points in the Julia set $J_\alpha$ gathers into $1$ point, which corresponds to the equivalence class.
	It can be considered as a parabolic implosion, because as the parameter $c$ perturbs from the root $r_M$ to the direction of angle $\alpha$, some periodic cycles pop out from the point.
	
	\begin{thm}\label{thm:shift_invariant}
		$\varphi_\alpha : \Sigma_2 / \sim~ \to \Sigma_2 / \sim$ is well-defined.
		Moreover it is order $2$ element and shift-invariant. 
	\end{thm}
	\begin{proof}
		It has order $2$ by the definition.
		Despite $\varphi_\alpha : \Sigma_2 \to \Sigma_2$ is multi-valued, but by the construction above the equivalence gathers all the points which cause the problem into one class.
		We check one by one.
		\begin{enumerate}
			\item Lift every sequence into a subset of extended angle $\mathbf{E}_\alpha$.
			Then $\varphi_\alpha$ is $I^{\overline{\alpha}} \circ (I^\alpha)^{-1}$.
			By the proposition \ref{prop:Julia_equiv_and_iti_seq}, $\varphi_\alpha$ is well defined on $\Sigma_2 / \sim$.
			\item By construction of $\varphi_\alpha$, it is of order $2$ element.
            In fact, $\varphi_\alpha^{-1} = \varphi_{\overline{\alpha}}$.
			\item Shift invariance is straightforward by the fact that $\partial \mathcal{L}_\alpha$ is invariant under angle doubling.
			As the shift map on $\Sigma_2$ corresponds to the angle doubling of the angle of $p$, $\varphi_\alpha$ commutes with one-sided shift.
		\end{enumerate}
	\end{proof}
	We end the section with the remark.
	Blanchard, Devaney, and Keen proved the following theorem in \cite{blanchard1991dynamics},
	\begin{thm}[BDK map]
		There is a surjective map from $\pi_1(\mathcal{S}_d)$ to $Aut(\Sigma_d, \sigma)$, a shift automorphism of $d$ symbols.
	\end{thm}
	The only nontrivial shift automorphism of $2$ symbol is a $0$-$1$ swap, aformentioned earlier in the introduction.
	So each $\varphi_\alpha$ is not a shift automorphism of $\Sigma_2$.
    With the equivalence in definition \ref{defn:equivalence_relation_on_Sigma_2}, we now get a shift invariance.

    However, there is another shift invariance at the subset of $\Sigma_2$.
	Suppose we ignore the prefix part.
    Then it became a $0v^\alpha$-$1v^\alpha$ sequence, and hence commute with $\sigma^{\circ \PER(\alpha)}$, a $\PER(\alpha)$ times shift.
	\emph{i.e.,} if $0$ sends to $0v^\alpha$ and $1$ sends to $1v^\alpha$ then it behaves like a shift map and the action of $\varphi_\alpha$ became similar to the nontrivial shift automorphism of $2$ symbols.
	
	With the Douady's angle formula \ref{lem:douady_angle_tuning_formula}, if we restrict our interest on the Fatou component which contains $0$, the $0$-$1$ symbol swap is translated into $0v^\alpha$-$1v^\alpha$ swap after the tuning.
	We will focus on this phenomenon and construct a big mapping class of $\Mod(\mathbb{C} - \{ \text{Cantor set}\})$ or $\Mod(S^2 - \{ \text{Cantor set}\})$ in the upcoming paper.
	
	\newpage
	\appendix
		\section{Symbolic dynamics and puzzle pieces} \label{appendix:A-sym_dyn}
		In this section we only care about the narrow components unless specified.
		The main purpose of the appendix is to construct the dynamical graph, which is analogous to that of Atela's paper \cite{atela1993mandelbrot}.
        Such dynamical graph is built up by puzzle pieces of the unit disc, divided by chords of the orbit portrait of $\alpha$, with some additional chords.
        Speaking in advance, such construction is valid for non-narrow cases also, except the canonical indexing the puzzle.
		
		\subsection{Markov puzzle pieces}
		In section \ref{sec:3-char_arcs}, we prove the necessary and sufficient condition between narrowness and no nestings of $I = A_1$.
		This fact enable us to number the puzzle piece in a canonical way.
		
		Suppose $I = (\alpha, \overline{\alpha}), \alpha < \overline{\alpha}$ is a narrow arc or period $n$ and $I = A_1, A_2, \cdots, A_n$ as defined earlier.
		Denote $a_i$ as a chord whose endpoints agree with those of $A_i$'s.
		We add $a_0$ and $S$, which is an antipodal chord of $a_n$ and a diameter $\frac{\alpha}{2}\frac{\alpha+1}{2}$.
		They separate the unit disc into $(n+2)$ or $(n+3)$ pieces if $I$ is satellite or primitive, respectively.
		(Recall that if $I$ is narrow, $I$ is satellite if and only if $I$ is attached on the main cardioid $M_0$.
		In this case we get $n+2$ puzzle pieces.)
		
		We number each piece as follows.
        We give an example in figure \ref{fig:puzzle_piece} also.
		
		\begin{enumerate}
			\item There are two puzzle pieces, one enveloped by $a_n, S$ and the other enveloped by $a_0, S$.
			We denote each as $\Pi_0^0$, $\Pi_0^1$.
			\item Suppose $a_i$ is the innermost arc.
			We number the puzzle piece whose boundary is a union of $a_i$ and the arc $A_i$ with $\Pi_i$.
			\item Suppose $A_{i_1} \subset A_{i_2} \subset \cdots \subset A_{i_k}$ with $i_k \neq n$.
			We denote the puzzle piece between $A_{i_j}$ and $A_{i_{j+1}}$ as $\Pi_{i_{j+1}}$.
			\item Finally there are $2$ pieces left, one next to $a_n$ and the other next to $a_0$.
			We call them as $\Pi_U$ and $\Pi_D$, respectively.
			Note that if $I$ is satellite, then $\Pi_U$ became a polygon gap, and the number of sides are equal to the period of $I$.
		\end{enumerate} 
		
		\begin{figure}[h]
		    \centering
		    \includegraphics[width=.5\textwidth]{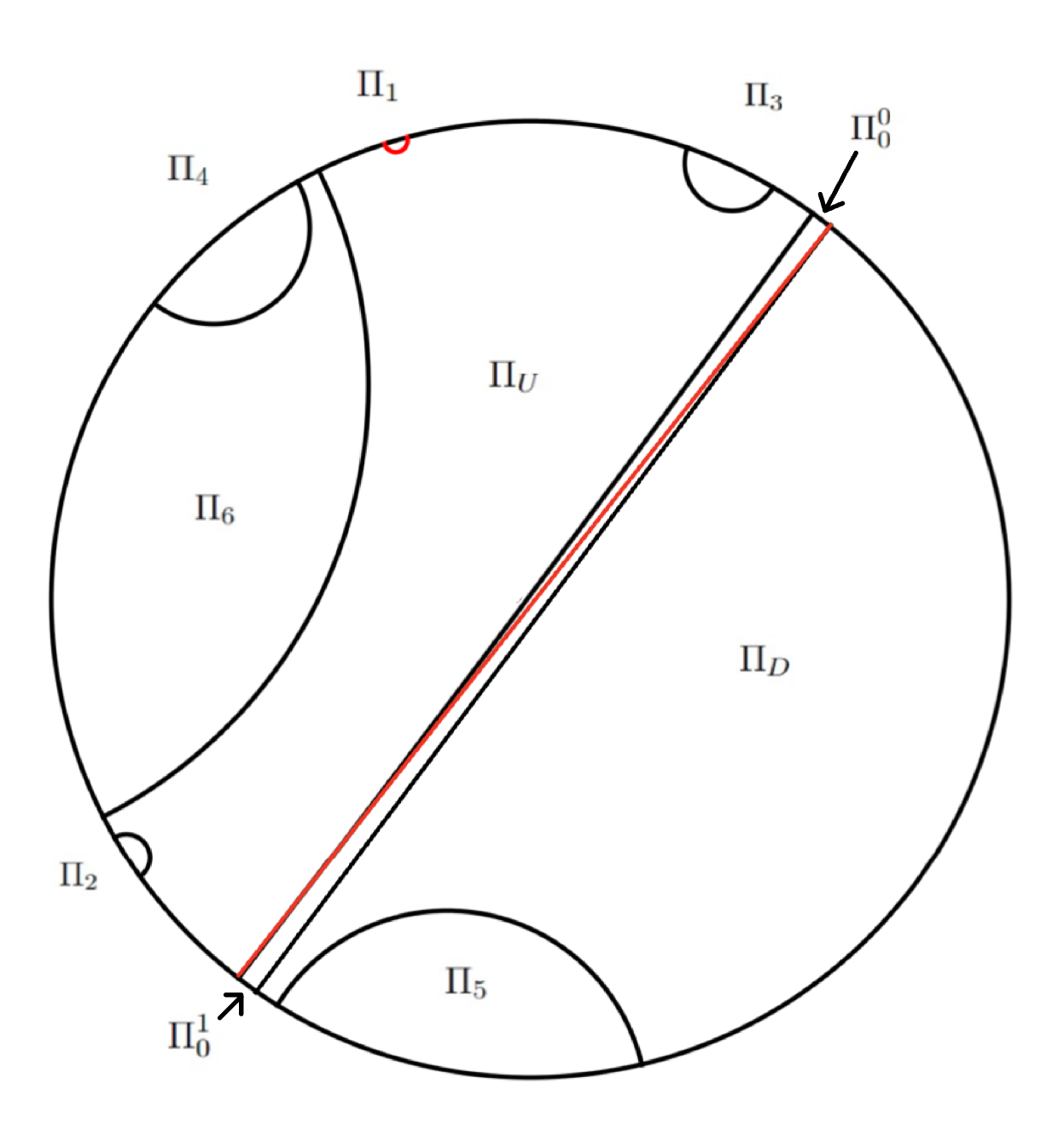}
		    \caption{Puzzle pieces of $I = \left(\frac{37}{127}, \frac{38}{127}\right)$.}
		    \label{fig:puzzle_piece}
		\end{figure}
		
		The angle doubling map $h : S^1 \to S^1$ maps each puzzle piece to some set of puzzle pieces.
		We draw a directed graph as follows.
		\begin{itemize}
			\item Vertices/states : Each puzzle pieces.
			\item Edges : a directed edge from $\Pi$ to $\Pi'$ if $\Pi' \subset h(\Pi)$.
		\end{itemize}
		Note that $\Pi_0^0$ and $\Pi_0^1$ both sends to $\Pi_1$.
		We call this as a transition graph $\mathcal{G}_\alpha$ associated to the angle $\alpha$.
		A path in $\mathcal{G}_\alpha$ is a one-sided sequence of states whose every consecutive $2$ states is connected by an edge compatible to its direction.
		\begin{lem}\label{lem:transition_graph_is_unique_for_char_arc}
			$\mathcal{G}_\alpha$ and $\mathcal{G}_{\overline{\alpha}}$ are isomorphic.
		\end{lem}
		\begin{proof}
			Since $\alpha$ and $\overline{\alpha}$ are the endpoints of $I$, every puzzle is identical except $\Pi_0^0$ and $\Pi_0^1$.
			But those are the preimages of $\Pi_1$, so the outgoing edge from $\Pi_0^*$ are same, both go to $\Pi_1$.
			Also the ingoing edges are also identical.
			Suppose not.
			Then there is a puzzle piece $\Pi$ whose image contains $\Pi_0^0$ but not $\Pi_0^1$.
			It implies that the boundary of $h(\Pi)$ contains $S$.
			Since the boundary of any pieces can admit $A_i$'s or $a_i$'s, and $h$ sends boundary to boundary, $S$ cannot be the boundary.
			Thus, every ingoing and outgoing edge of both $\Pi_0^0$ and $\Pi_0^1$ is identical and hence we can consider it as a unified one state $\Pi_0$.
			Define transition graph $\mathcal{G}'$ of puzzle pieces with $\Pi_0$ instead of $\Pi_0^0, \Pi_0^1$.
			Obviously same $\mathcal{G}'$ is obtained from $\alpha$ or $\overline{\alpha}$, we get the same transition graph.
		\end{proof}
		The difference between those two transition graph appears when we assign a symbol $\in\{0,1\}$ to each puzzle piece.
		We assign each puzzle piece with a symbol as we produce an itinerary sequence.
        \emph{i.e.,} a puzzle piece $\Pi$ is assigned by $0$ if $\Pi$ is contained in $A_\alpha$, $1$ if it is contained in $B_\alpha$.
		Note that $\Pi_0^0, \Pi_1$ and $\Pi_U$ are always marked by $0$, $\Pi_0^1, \Pi_D$ is $1$ on the other side.
		If we only focus on the arc boundaries of each puzzle pieces, then the arc boundary of $\Pi_0^{0, \alpha}$ is same with that of $\Pi_0^{1, \overline{\alpha}}$.
		
		The symbol assigning map $\{\text{puzzle pieces}\} \mapsto \{0,1\}$ induces a map $\{\text{Paths in }\mathcal{G}_\alpha\} \to \Sigma_2$.
		\emph{i.e.,} Each path will give a description of itinerary sequences. 
		In fact, every sequence in $\Sigma_2$ is realized.
		
		\begin{figure}[h]
		    \centering
		    \includegraphics[width=.6\textwidth]{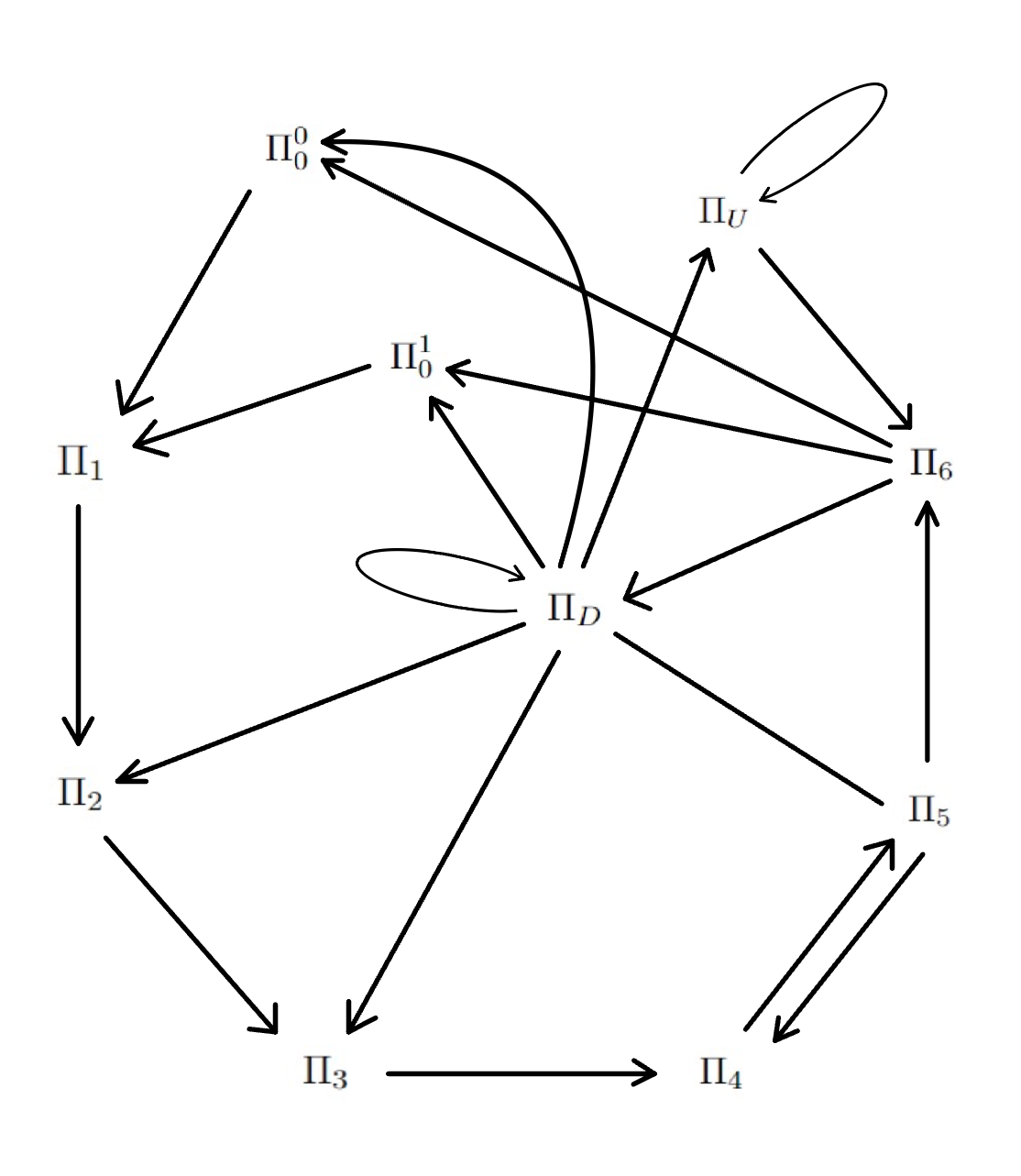}
		    \caption{The transition graph $\mathcal{G}_\alpha$ for $\alpha = \frac{37}{127}$}
		    \label{fig:transition_graph}
		\end{figure}
		
		\begin{thm}\label{thm:path_onto_sequences}
			Let $\mathbb{P}$ be the collection of all paths of the transition graph $\mathcal{G}_\alpha$.
			Then the symbol assigning map is a surjective map $\mathbb{P} \twoheadrightarrow \Sigma_2$.
		\end{thm}
		
		We will prove this by showing each sequence is a $I^\alpha(\theta)$ for some $\theta\in \mathbf{E}_\alpha$.
		
		\begin{defn}\label{defn:D_x's}
			Let $x\in\{0,1\}^*$ be a $0$-$1$ word and $\alpha\in S^1$ be nonzero periodic.
			$D_x^\alpha$ is defined as a subset of $\mathbf{E}_\alpha$ that $\{\theta \in \mathbf{E}_\alpha ~|~ I^\alpha(\theta) \text{ starts with }x\}$.
			Also one can define for any sequence $\mathbf{x} = \{x_1, x_2, \cdots\} \in \{0,1\}^\mathbb{N}$ as $D_\mathbf{x} := \bigcap_{n = 1}^\infty D_{x_1\cdots x_n}$.
		\end{defn}
			
		\begin{lem}\label{lem:D_x's}
			\begin{enumerate}			
				\item If $x$ is $\alpha$-regular, then $D_x^\alpha = D_{x0}^\alpha \cup D_{x1}^\alpha$ and $D_{x0}^\alpha \cap D_{x1}^\alpha = S^\alpha_x$.
				\item If $x$ is not $\alpha$-regular, then $D_x^\alpha = D_{x0}^\alpha \cup D_{x1}^\alpha \cup \Delta_x^\alpha$.
				In this case, $D_{x0}^\alpha \cap \Delta_x^\alpha = S^{\alpha, 0}_x$ and $D_{x1}^\alpha \cap \Delta_x^\alpha = S^{\alpha, 1}_x$.
			\end{enumerate}
		\end{lem}
		\begin{proof}
			This is a theorem 2.15 in \cite{keller2007invariant}.
		\end{proof}
		The above lemma implies that such $D_x$ has a boundary of chords, and as $|x| \to \infty$ it splits into $2$ pieces by another chord generated by the long chord $\frac{\alpha}{2}\frac{(\alpha+1)}{2}$.
		Hence we get the following corollary.		
		\begin{cor}\label{cor:puzzles_contains_D_x}
			Choose arbitrary $\mathbf{x}\in \{0,1\}^\mathbb{N}$.
			For any $\theta_1, \theta_2 \in D_\mathbf{x}$, $\theta_1\theta_2$ is unlinked with the orbit portrait of $\alpha$.
		\end{cor}
		\begin{proof}
			$\partial\mathcal{L}_\alpha$ is generated by $\alpha\overline{\alpha}$, hence it contains every chord in the orbit portrait of $\alpha$.
			By lemma \ref{lem:D_x's}, the boundary of each $D_x$ consists of chords of $\mathcal{L}_\alpha$ for any word $x \in \{0,1\}^*$.
			Taking the limit, $D_\mathbf{x}$ is bounded by chords in $\partial\mathcal{L}_\alpha$, therefore the convex hull of them is unlinked with the orbit portrait.		
		\end{proof}
	
	\begin{proof}
		(Proof of theorem \ref{thm:path_onto_sequences}) Choose any sequence $x \in \Sigma_2$.
		Then by corollary \ref{cor:puzzles_contains_D_x}, it is either properly contained in a puzzle or $\partial D_x$ comprises boundaries of some puzzle pieces.
		Choose any angle $\theta$ in $D_x$.
		Note that $\theta$ is an extended angle.
		The path on transition graph can be obtained by keeping track on which puzzle piece the angle $h^{\circ m}(\theta)$ belongs to.
	\end{proof}
		
		The puzzle piece $\Pi_0$ admits a preimage of $A_1$ as its boundary.
		Note that the itinerary sequence symbol changes only at this region.
		In other words, we have the following rule how the itinerary sequence changes when $\alpha$ varies to $\overline{\alpha}$.
		
		\begin{thm}[Itinerary sequence rule]\label{thm:iti_seq_rule}
			For $\theta \in \mathbf{E}_\alpha$, $\varphi_\alpha : I^\alpha(\theta) \mapsto I^{\overline{\alpha}}(\theta)$ by the following rule.
			\begin{enumerate}
				\item Lift $I^{\alpha}(\theta)$ to a path in $\mathcal{G}_\alpha$.
				\item Consider it as a path in $\mathcal{G}_{\overline{\alpha}}$.
				\item Re-assign the path with respect to $\overline{\alpha}$.
			\end{enumerate}
			By the construction, $\varphi_\alpha$ is order $2$ element.
		\end{thm}
	
		Final remark is that the construction we introduced here is not limited in the narrow cases, as we aforementioned in the beginning.
        The only limitation is that we could not find any canonical ways to number/index each state/puzzle piece for non-narrow cases.
		However, every theorem in the appendix \ref{appendix:A-sym_dyn} still holds for the non-narrow cases, as analogous as the Atela's work in \cite{atela1993mandelbrot}.
		Nonetheless, it still exists that the ambiguity of $\varphi_\alpha$ and we cannot explicitly find how the sequence varies if we have an input only.
	
		\section{Tables of not simply renormalizable, non-narrow arcs of non-prime period}\label{appendix:B-3non}
		Renormalizability can be detected by its internal address, which is first proposed by Dierk Schleicher in his paper \cite{schleicher2017internal}.
		We introduce his theorem about relation between renormalizability and internal address.
		\begin{thm}[Proposition 4.7 in \cite{schleicher2017internal}]
			Suppose the characteristic arc $I$ admits $[S_1, S_2, \cdots, S_n]$ as its internal address.
			$I$ is (simply) renormalizable if and only if there exists $S_i$ such that $S_i~|~S_{j}$ for all $i \leq j \leq n$.
		\end{thm}		
		Here we list all characteristic arcs of non-prime period, which are not narrow, and non-renormalizable, up to period $10$.
		Note that there are only one non narrow characteristic arc of period $4$, which is $(2/5, 3/5)$ and it is a basilica tuned by basilica. (Hence renormalizable.)
		
		We give a kneading sequence and internal address for each characteristic arc.
		Since all arcs are symmetric by $\frac{1}{2}$ (if $(a,b)$ is characteristic arc, then $\left(\frac{1}{2}-b, \frac{1}{2}-a\right)$ is also a characteristic arc), we only provide one of each symmetric case.
		The tuple in the kneading sequence column implies its $v^\alpha$.
		\emph{i.e.,} $I^\alpha(\alpha) = \overline{v^\alpha *}$.

        We also provide the ratio of non-narrow arcs among all not simply renormalizable characteristic arcs of each period.
        Here the column narrow and non-narrow are the number of total narrow and non-narrow arcs regardless of renormalizablity.
        
        \begin{table}[h]
    		\tiny
    		\begin{tabular}{|c|r|r|r| |r|r|r|}
    			\hline
                 & \multicolumn{3}{|c||}{All arcs} & \multicolumn{3}{|c|}{Non-renormalizable arcs} \\
                \hline
                Period & narrow & non-narrow & total & non-narrow & total & ratio \\
                \hline
                4 & 5 & 1 & 6 & 0 & 5 & 0\\
                6 & 19 & 8 & 27 & 2 & 21 & 0.095238\\
                8 & 83 & 37 & 120 & 26 & 109 & 0.238532\\
                9 & 173 & 79 & 252 & 70 & 243 & 0.288065\\
                10 & 325 & 170 & 495 & 140 & 465 & 0.301075\\
                12 & 1303 & 707 & 2010 & 626 & 1929 & 0.324520\\
                14 & 5231 & 2896 & 8127 & 2770 & 8001 & 0.346206\\
                15 & 10527 & 5838 & 16365 & 5748 & 16275 & 0.353179\\
    			\hline
    		\end{tabular}
            \caption{The ratio measures how much portion the number of non-narrow arcs occupies among the total number of non-simply renormalizable arcs.}
            \label{tab:ratio}
    	\end{table}
     
        As one can see in the table \ref{tab:ratio}, the ratio of non-narrow not simply renormalizable arcs keep increasing as the period does, up to $15$.
        It is natural to ask whether this ratio converges as nonprime periods diverge.
        We do not have a theoretical answer yet.
        
		\begin{table}[h]
			\tiny
		\begin{adjustbox}{width=\columnwidth,center}
		\begin{tabular}{|l||l|l| |l||l|l|}
            \hline
			\multicolumn{6}{|l|}{Period = 6} \\
			\hline
			Characteristic arc & Kneading sequence & Internal address & Characteristic arc & Kneading sequence & Internal address \\
			\hline
			29/63, 34/63 & [0, 1, 1, 0, 0] & [1, 2, 3, 5, 6] & 10/21, 11/21 & [0, 1, 1, 1, 0] & [1, 2, 3, 4, 6] \\
			\hline
			\hline
			\multicolumn{6}{|l|}{Period = 8} \\
			\hline
			Characteristic arc & Kneading sequence & Internal address & Characteristic arc & Kneading sequence & Internal address\\
			\hline
			41/255, 10/51 & [0, 0, 1, 0, 0, 0, 0] & [1, 3, 6, 8] & 116/255, 139/255 & [0, 1, 1, 0, 0, 0, 1] & [1, 2, 3, 5, 8] \\
			14/85, 49/255 & [0, 0, 1, 0, 1, 0, 0] & [1, 3, 5, 8] & 39/85, 46/85 & [0, 1, 1, 0, 0, 0, 0] & [1, 2, 3, 5, 7, 8] \\
			19/85, 22/85 & [0, 0, 1, 1, 0, 0, 0] & [1, 3, 4, 7, 8] & 118/255, 137/255 & [0, 1, 1, 0, 0, 1, 0] & [1, 2, 3, 5, 6, 8] \\
			58/255, 13/51 & [0, 0, 1, 1, 1, 0, 0] & [1, 3, 4, 5, 8] & 121/255, 134/255 & [0, 1, 1, 1, 0, 1, 0] & [1, 2, 3, 4, 7, 8] \\
			31/85, 98/255 & [0, 1, 0, 1, 1, 0, 0] & [1, 2, 5, 7, 8] & 122/255, 133/255 & [0, 1, 1, 1, 0, 0, 0] & [1, 2, 3, 4, 6, 8] \\
			94/255, 97/255 & [0, 1, 0, 1, 1, 1, 0] & [1, 2, 5, 6, 8] & 41/85, 44/85 & [0, 1, 1, 1, 0, 0, 1] & [1, 2, 3, 4, 6, 7, 8] \\
			109/255, 146/255 & [0, 1, 0, 0, 1, 0, 0] & [1, 2, 4, 5, 7, 8] & 124/255, 131/255 & [0, 1, 1, 1, 1, 0, 1] & [1, 2, 3, 4, 5, 8] \\
			22/51, 113/255 & [0, 1, 1, 0, 1, 1, 0] & [1, 2, 3, 8] & 25/51, 26/51 & [0, 1, 1, 1, 1, 0, 0] & [1, 2, 3, 4, 5, 7, 8] \\
			38/85, 47/85 & [0, 1, 1, 0, 1, 0, 0] & [1, 2, 3, 6, 8] & 42/85, 43/85 & [0, 1, 1, 1, 1, 1, 0] & [1, 2, 3, 4, 5, 6, 8] \\
			23/51, 28/51 & [0, 1, 1, 0, 1, 0, 1] & [1, 2, 3, 6, 7, 8] & & & \\
			\hline
		\end{tabular}
		\end{adjustbox}
        \caption{Period  = $6, 8$}
		\end{table}
		
		\begin{table}[h]
			\tiny
		\begin{adjustbox}{width=\columnwidth,center}
		\begin{tabular}{|l||l|l| |l||l|l| }
			\hline
			\multicolumn{6}{|l|}{Period = 9} \\
			\hline
			Characteristic arc & Kneading sequence & Internal address & Characteristic arc & Kneading sequence & Internal address\\
			\hline
			7/73, 66/511 & [0, 0, 0, 1, 0, 0, 0, 0] & [1, 4, 8, 9] & 221/511, 226/511 & [0, 1, 1, 0, 1, 1, 0, 0] & [1, 2, 3, 8, 9] \\
			50/511, 65/511 & [0, 0, 0, 1, 1, 0, 0, 0] & [1, 4, 5, 9] & 222/511, 225/511 & [0, 1, 1, 0, 1, 1, 1, 0] & [1, 2, 3, 7, 9] \\
			89/511, 14/73 & [0, 0, 1, 0, 1, 0, 0, 0] & [1, 3, 5, 8, 9] & 229/511, 282/511 & [0, 1, 1, 0, 1, 0, 0, 0] & [1, 2, 3, 6, 8, 9] \\
			90/511, 97/511 & [0, 0, 1, 0, 1, 1, 0, 0] & [1, 3, 5, 6, 9] & 230/511, 281/511 & [0, 1, 1, 0, 1, 0, 1, 0] & [1, 2, 3, 6, 7, 9] \\
			15/73, 114/511 & [0, 0, 1, 1, 0, 0, 0, 0] & [1, 3, 4, 7, 9] & 233/511, 278/511 & [0, 1, 1, 0, 0, 0, 1, 0] & [1, 2, 3, 5, 8, 9] \\
			106/511, 113/511 & [0, 0, 1, 1, 0, 1, 0, 0] & [1, 3, 4, 6, 9] & 234/511, 277/511 & [0, 1, 1, 0, 0, 0, 0, 0] & [1, 2, 3, 5, 7, 9] \\
			115/511, 132/511 & [0, 0, 1, 1, 0, 0, 0, 1] & [1, 3, 4, 7, 8, 9] & 235/511, 276/511 & [0, 1, 1, 0, 0, 0, 0, 1] & [1, 2, 3, 5, 7, 8, 9] \\
			116/511, 131/511 & [0, 0, 1, 1, 1, 0, 0, 1] & [1, 3, 4, 5, 9] & 236/511, 275/511 & [0, 1, 1, 0, 0, 1, 0, 1] & [1, 2, 3, 5, 6, 9] \\
			121/511, 130/511 & [0, 0, 1, 1, 1, 0, 0, 0] & [1, 3, 4, 5, 8, 9] & 237/511, 274/511 & [0, 1, 1, 0, 0, 1, 0, 0] & [1, 2, 3, 5, 6, 8, 9] \\
			122/511, 129/511 & [0, 0, 1, 1, 1, 1, 0, 0] & [1, 3, 4, 5, 6, 9] & 34/73, 39/73 & [0, 1, 1, 0, 0, 1, 1, 0] & [1, 2, 3, 5, 6, 7, 9] \\
			173/511, 178/511 & [0, 1, 0, 1, 0, 1, 0, 0] & [1, 2, 8, 9] & 241/511, 270/511 & [0, 1, 1, 1, 0, 1, 1, 0] & [1, 2, 3, 4, 8, 9] \\
			174/511, 177/511 & [0, 1, 0, 1, 0, 1, 1, 0] & [1, 2, 7, 9] & 242/511, 269/511 & [0, 1, 1, 1, 0, 1, 0, 0] & [1, 2, 3, 4, 7, 9] \\
			181/511, 202/511 & [0, 1, 0, 1, 0, 0, 0, 0] & [1, 2, 6, 8, 9] &	243/511, 268/511 & [0, 1, 1, 1, 0, 1, 0, 1] & [1, 2, 3, 4, 7, 8, 9] \\
			26/73, 185/511 & [0, 1, 0, 1, 1, 0, 1, 0] & [1, 2, 5, 9] & 244/511, 267/511 & [0, 1, 1, 1, 0, 0, 0, 1] & [1, 2, 3, 4, 6, 9] \\
			186/511, 197/511 & [0, 1, 0, 1, 1, 0, 0, 0] & [1, 2, 5, 7, 9] & 35/73, 38/73 & [0, 1, 1, 1, 0, 0, 0, 0] & [1, 2, 3, 4, 6, 8, 9] \\
			27/73, 194/511 & [0, 1, 0, 1, 1, 1, 0, 0] & [1, 2, 5, 6, 8, 9] & 246/511, 265/511 & [0, 1, 1, 1, 0, 0, 1, 0] & [1, 2, 3, 4, 6, 7, 9] \\
			190/511, 193/511 & [0, 1, 0, 1, 1, 1, 1, 0] & [1, 2, 5, 6, 7, 9] & 247/511, 264/511 & [0, 1, 1, 1, 0, 0, 1, 1] & [1, 2, 3, 4, 6, 7, 8, 9] \\
			198/511, 201/511 & [0, 1, 0, 1, 0, 0, 1, 0] & [1, 2, 6, 7, 9] & 248/511, 263/511 & [0, 1, 1, 1, 1, 0, 1, 1] & [1, 2, 3, 4, 5, 9] \\
			205/511, 30/73 & [0, 1, 0, 0, 0, 1, 0, 0] & [1, 2, 4, 9] & 249/511, 262/511 & [0, 1, 1, 1, 1, 0, 1, 0] & [1, 2, 3, 4, 5, 8, 9] \\
			206/511, 209/511 & [0, 1, 0, 0, 0, 1, 1, 0] & [1, 2, 4, 7, 9] & 250/511, 261/511 & [0, 1, 1, 1, 1, 0, 0, 0] & [1, 2, 3, 4, 5, 7, 9] \\
			213/511, 298/511 & [0, 1, 0, 0, 0, 0, 0, 0] & [1, 2, 4, 6, 8, 9] & 251/511, 260/511 & [0, 1, 1, 1, 1, 0, 0, 1] & [1, 2, 3, 4, 5, 7, 8, 9] \\
			214/511, 297/511 & [0, 1, 0, 0, 0, 0, 1, 0] & [1, 2, 4, 6, 7, 9] & 36/73, 37/73 & [0, 1, 1, 1, 1, 1, 0, 1] & [1, 2, 3, 4, 5, 6, 9] \\
			31/73, 42/73 & [0, 1, 0, 0, 1, 0, 1, 0] & [1, 2, 4, 5, 9] & 253/511, 258/511 & [0, 1, 1, 1, 1, 1, 0, 0] & [1, 2, 3, 4, 5, 6, 8, 9] \\
			218/511, 293/511 & [0, 1, 0, 0, 1, 0, 0, 0] & [1, 2, 4, 5, 7, 9] & 254/511, 257/511 & [0, 1, 1, 1, 1, 1, 1, 0] & [1, 2, 3, 4, 5, 6, 7, 9] \\
			\hline
		\end{tabular}
		\end{adjustbox}
        \caption{Period  = $9$}
		\end{table}
	
		\begin{table}[h]
			\tiny
		\begin{adjustbox}{width=\columnwidth,center}
		\begin{tabular}{|l||l|l| |l||l|l| }
			\hline
			\multicolumn{6}{|l|}{Period = 10} \\
			\hline
			Characteristic arc & Kneading sequence & Internal address & Characteristic arc & Kneading sequence & Internal address\\
			\hline
			27/341, 98/1023 & [0, 0, 0, 1, 0, 0, 0, 0, 0] & [1, 4, 8, 10] & 443/1023, 452/1023 & [0, 1, 1, 0, 1, 1, 0, 0, 1] & [1, 2, 3, 8, 9, 10]\\
			82/1023, 97/1023 & [0, 0, 0, 1, 0, 1, 0, 0, 0] & [1, 4, 6, 10] & 148/341, 41/93 & [0, 1, 1, 0, 1, 1, 1, 0, 1] & [1, 2, 3, 7, 10] \\
			113/1023, 130/1023 & [0, 0, 0, 1, 1, 0, 0, 0, 0] & [1, 4, 5, 9, 10] & 445/1023, 150/341 & [0, 1, 1, 0, 1, 1, 1, 0, 0] & [1, 2, 3, 7, 9, 10] \\
			38/341, 43/341 & [0, 0, 0, 1, 1, 1, 0, 0, 0] & [1, 4, 5, 6, 10] & 446/1023, 449/1023 & [0, 1, 1, 0, 1, 1, 1, 1, 0] & [1, 2, 3, 7, 8, 10] \\
			51/341, 54/341 & [0, 0, 1, 0, 0, 1, 0, 0, 0] & [1, 3, 9, 10] & 457/1023, 566/1023 & [0, 1, 1, 0, 1, 0, 0, 1, 0] & [1, 2, 3, 6, 9, 10] \\
			14/93, 161/1023 & [0, 0, 1, 0, 0, 1, 1, 0, 0] & [1, 3, 7, 10] & 458/1023, 565/1023 & [0, 1, 1, 0, 1, 0, 0, 0, 0] & [1, 2, 3, 6, 8, 10] \\
			169/1023, 178/1023 & [0, 0, 1, 0, 1, 0, 0, 0, 0] & [1, 3, 5, 8, 10] & 153/341, 188/341 & [0, 1, 1, 0, 1, 0, 0, 0, 1] & [1, 2, 3, 6, 8, 9, 10] \\
			170/1023, 59/341 & [0, 0, 1, 0, 1, 0, 1, 0, 0] & [1, 3, 5, 7, 10] & 460/1023, 563/1023 & [0, 1, 1, 0, 1, 0, 1, 0, 1] & [1, 2, 3, 6, 7, 10] \\
			185/1023, 194/1023 & [0, 0, 1, 0, 1, 1, 0, 0, 0] & [1, 3, 5, 6, 9, 10] & 461/1023, 562/1023 & [0, 1, 1, 0, 1, 0, 1, 0, 0] & [1, 2, 3, 6, 7, 9, 10]\\
			2/11, 193/1023 & [0, 0, 1, 0, 1, 1, 1, 0, 0] & [1, 3, 5, 6, 7, 10] & 466/1023, 557/1023 & [0, 1, 1, 0, 0, 0, 1, 0, 0] & [1, 2, 3, 5, 8, 10] \\
			67/341, 274/1023 & [0, 0, 1, 0, 0, 0, 0, 0, 0] & [1, 3, 6, 9, 10] & 467/1023, 556/1023 & [0, 1, 1, 0, 0, 0, 1, 0, 1] & [1, 2, 3, 5, 8, 9, 10] \\
			202/1023, 91/341 & [0, 0, 1, 0, 0, 0, 1, 0, 0] & [1, 3, 6, 7, 10] & 156/341, 185/341 & [0, 1, 1, 0, 0, 0, 0, 0, 1] & [1, 2, 3, 5, 7, 10] \\
			19/93, 266/1023 & [0, 0, 1, 1, 0, 0, 1, 0, 0] & [1, 3, 4, 8, 10] & 469/1023, 554/1023 & [0, 1, 1, 0, 0, 0, 0, 0, 0] & [1, 2, 3, 5, 7, 9, 10] \\
			70/341, 265/1023 & [0, 0, 1, 1, 0, 0, 0, 0, 0] & [1, 3, 4, 7, 10] & 470/1023, 553/1023 & [0, 1, 1, 0, 0, 0, 0, 1, 0] & [1, 2, 3, 5, 7, 8, 10] \\
			211/1023, 76/341 & [0, 0, 1, 1, 0, 0, 0, 0, 1] & [1, 3, 4, 7, 9, 10] & 43/93, 50/93 & [0, 1, 1, 0, 0, 1, 0, 1, 0] & [1, 2, 3, 5, 6, 9, 10] \\
			212/1023, 227/1023 & [0, 0, 1, 1, 0, 1, 0, 0, 1] & [1, 3, 4, 6, 10] & 158/341, 183/341 & [0, 1, 1, 0, 0, 1, 0, 0, 0] & [1, 2, 3, 5, 6, 8, 10] \\
			7/33, 226/1023 & [0, 0, 1, 1, 0, 1, 0, 0, 0] & [1, 3, 4, 6, 9, 10] & 475/1023, 548/1023 & [0, 1, 1, 0, 0, 1, 0, 0, 1] & [1, 2, 3, 5, 6, 8, 9, 10] \\
			218/1023, 75/341 & [0, 0, 1, 1, 0, 1, 1, 0, 0] & [1, 3, 4, 6, 7, 10] & 476/1023, 547/1023 & [0, 1, 1, 0, 0, 1, 1, 0, 1] & [1, 2, 3, 5, 6, 7, 10] \\
			233/1023, 22/93 & [0, 0, 1, 1, 1, 0, 0, 0, 0] & [1, 3, 4, 5, 8, 10] & 159/341, 182/341 & [0, 1, 1, 0, 0, 1, 1, 0, 0] & [1, 2, 3, 5, 6, 7, 9, 10] \\
			78/341, 241/1023 & [0, 0, 1, 1, 1, 0, 1, 0, 0] & [1, 3, 4, 5, 7, 10] & 478/1023, 481/1023 & [0, 1, 1, 1, 0, 1, 1, 1, 0] & [1, 2, 3, 4, 10] \\
			81/341, 260/1023 & [0, 0, 1, 1, 1, 0, 0, 0, 1] & [1, 3, 4, 5, 8, 9, 10] & 482/1023, 541/1023 & [0, 1, 1, 1, 0, 1, 1, 0, 0] & [1, 2, 3, 4, 8, 10] \\
			244/1023, 259/1023 & [0, 0, 1, 1, 1, 1, 0, 0, 1] & [1, 3, 4, 5, 6, 10] & 161/341, 180/341 & [0, 1, 1, 1, 0, 1, 1, 0, 1] & [1, 2, 3, 4, 8, 9, 10] \\
			83/341, 86/341 & [0, 0, 1, 1, 1, 1, 0, 0, 0] & [1, 3, 4, 5, 6, 9, 10] & 44/93, 49/93 & [0, 1, 1, 1, 0, 1, 0, 0, 1] & [1, 2, 3, 4, 7, 10] \\
			250/1023, 257/1023 & [0, 0, 1, 1, 1, 1, 1, 0, 0] & [1, 3, 4, 5, 6, 7, 10] & 485/1023, 538/1023 & [0, 1, 1, 1, 0, 1, 0, 0, 0] & [1, 2, 3, 4, 7, 9, 10] \\
			281/1023, 290/1023 & [0, 0, 1, 0, 0, 1, 0, 0, 0] & [1, 3, 9, 10] & 162/341, 179/341 & [0, 1, 1, 1, 0, 1, 0, 1, 0] & [1, 2, 3, 4, 7, 8, 10] \\
			94/341, 289/1023 & [0, 0, 1, 0, 0, 1, 1, 0, 0] & [1, 3, 7, 10] & 487/1023, 536/1023 & [0, 1, 1, 1, 0, 1, 0, 1, 1] & [1, 2, 3, 4, 7, 8, 9, 10] \\
			349/1023, 118/341 & [0, 1, 0, 1, 0, 1, 1, 0, 0] & [1, 2, 7, 9, 10] & 488/1023, 535/1023 & [0, 1, 1, 1, 0, 0, 0, 1, 1] & [1, 2, 3, 4, 6, 10] \\
			350/1023, 353/1023 & [0, 1, 0, 1, 0, 1, 1, 1, 0] & [1, 2, 7, 8, 10] & 163/341, 178/341 & [0, 1, 1, 1, 0, 0, 0, 1, 0] & [1, 2, 3, 4, 6, 9, 10] \\
			365/1023, 370/1023 & [0, 1, 0, 1, 1, 0, 1, 0, 0] & [1, 2, 5, 9, 10] & 490/1023, 533/1023 & [0, 1, 1, 1, 0, 0, 0, 0, 0] & [1, 2, 3, 4, 6, 8, 10] \\
			122/341, 123/341 & [0, 1, 0, 1, 1, 0, 1, 1, 0] & [1, 2, 5, 8, 10] & 491/1023, 532/1023 & [0, 1, 1, 1, 0, 0, 0, 0, 1] & [1, 2, 3, 4, 6, 8, 9, 10] \\
			373/1023, 394/1023 & [0, 1, 0, 1, 1, 0, 0, 0, 0] & [1, 2, 5, 7, 9, 10] & 164/341, 177/341 & [0, 1, 1, 1, 0, 0, 1, 0, 1] & [1, 2, 3, 4, 6, 7, 10] \\
			34/93, 377/1023 & [0, 1, 0, 1, 1, 1, 0, 1, 0] & [1, 2, 5, 6, 10] & 493/1023, 530/1023 & [0, 1, 1, 1, 0, 0, 1, 0, 0] & [1, 2, 3, 4, 6, 7, 9, 10] \\
			126/341, 389/1023 & [0, 1, 0, 1, 1, 1, 0, 0, 0] & [1, 2, 5, 6, 8, 10] & 494/1023, 529/1023 & [0, 1, 1, 1, 0, 0, 1, 1, 0] & [1, 2, 3, 4, 6, 7, 8, 10] \\
			127/341, 386/1023 & [0, 1, 0, 1, 1, 1, 1, 0, 0] & [1, 2, 5, 6, 7, 9, 10] & 497/1023, 526/1023 & [0, 1, 1, 1, 1, 0, 1, 1, 0] & [1, 2, 3, 4, 5, 9, 10] \\
			382/1023, 35/93 & [0, 1, 0, 1, 1, 1, 1, 1, 0] & [1, 2, 5, 6, 7, 8, 10] & 166/341, 175/341 & [0, 1, 1, 1, 1, 0, 1, 0, 0] & [1, 2, 3, 4, 5, 8, 10] \\
			130/341, 131/341 & [0, 1, 0, 1, 1, 0, 0, 1, 0] & [1, 2, 5, 7, 8, 10]  & 499/1023, 524/1023 & [0, 1, 1, 1, 1, 0, 1, 0, 1] & [1, 2, 3, 4, 5, 8, 9, 10] \\
			397/1023, 134/341 & [0, 1, 0, 1, 0, 0, 1, 0, 0] & [1, 2, 6, 7, 9, 10] & 500/1023, 523/1023 & [0, 1, 1, 1, 1, 0, 0, 0, 1] & [1, 2, 3, 4, 5, 7, 10] \\
			398/1023, 401/1023 & [0, 1, 0, 1, 0, 0, 1, 1, 0] & [1, 2, 6, 7, 8, 10] & 167/341, 174/341 & [0, 1, 1, 1, 1, 0, 0, 0, 0] & [1, 2, 3, 4, 5, 7, 9, 10] \\
			413/1023, 38/93 & [0, 1, 0, 0, 0, 1, 1, 0, 0] & [1, 2, 4, 7, 9, 10] & 502/1023, 521/1023 & [0, 1, 1, 1, 1, 0, 0, 1, 0] & [1, 2, 3, 4, 5, 7, 8, 10] \\
			138/341, 139/341 & [0, 1, 0, 0, 0, 1, 1, 1, 0] & [1, 2, 4, 7, 8, 10] & 503/1023, 520/1023 & [0, 1, 1, 1, 1, 0, 0, 1, 1] & [1, 2, 3, 4, 5, 7, 8, 9, 10] \\
			430/1023, 433/1023 & [0, 1, 0, 0, 1, 0, 1, 1, 0] & [1, 2, 4, 5, 8, 10] & 168/341, 173/341 & [0, 1, 1, 1, 1, 1, 0, 1, 1] & [1, 2, 3, 4, 5, 6, 10] \\
			437/1023, 586/1023 & [0, 1, 0, 0, 1, 0, 0, 0, 0] & [1, 2, 4, 5, 7, 9, 10] & 505/1023, 518/1023 & [0, 1, 1, 1, 1, 1, 0, 1, 0] & [1, 2, 3, 4, 5, 6, 9, 10] \\
			146/341, 195/341 & [0, 1, 0, 0, 1, 0, 0, 1, 0] & [1, 2, 4, 5, 7, 8, 10] & 46/93, 47/93 & [0, 1, 1, 1, 1, 1, 0, 0, 0] & [1, 2, 3, 4, 5, 6, 8, 10] \\
			147/341, 454/1023 & [0, 1, 1, 0, 1, 1, 0, 1, 0] & [1, 2, 3, 9, 10] & 169/341, 172/341 & [0, 1, 1, 1, 1, 1, 0, 0, 1] & [1, 2, 3, 4, 5, 6, 8, 9, 10] \\
			442/1023, 151/341 & [0, 1, 1, 0, 1, 1, 0, 0, 0] & [1, 2, 3, 8, 10] & 508/1023, 515/1023 & [0, 1, 1, 1, 1, 1, 1, 0, 1] & [1, 2, 3, 4, 5, 6, 7, 10] \\
			509/1023, 514/1023 & [0, 1, 1, 1, 1, 1, 1, 0, 0] & [1, 2, 3, 4, 5, 6, 7, 9, 10] & 170/341, 171/341 & [0, 1, 1, 1, 1, 1, 1, 1, 0] & [1, 2, 3, 4, 5, 6, 7, 8, 10] \\
			\hline
		\end{tabular}
		\end{adjustbox}
        \caption{Period  = $10$}
		\end{table}
    
	\clearpage
	\bibliographystyle{alpha} 
	\bibliography{refs}
	
\end{document}